\documentclass[11pt,reqno]{amsart}
\usepackage{a4wide} 
\usepackage{graphicx}
\usepackage{hyperref}
\usepackage{amsmath,amsopn,amssymb,amsfonts,stmaryrd}
\usepackage{verbatim}
\usepackage{amsthm}
\usepackage{mathtools}
\usepackage{color}
\usepackage{simplewick}
\usepackage{enumitem}
\usepackage[framemethod=TikZ]{mdframed}
\usepackage{bbm}
\usepackage{mathrsfs}
\usepackage{booktabs}
\usepackage{caption}
\newcommand{\IR}{{\mathbb R}}
\newcommand{\IC}{{\mathbb C}}
\newcommand{\IP}{{\mathbb P}}
\newcommand{\IZ}{{\mathbb Z}}
\newcommand{\IN}{{\mathbb N}}
\newcommand{\IQ}{{\mathbb Q}}
\newcommand{\IH}{{\mathbb H}}

\newcommand{\U}{\mathrm{U}}

\newcommand{\SL}{\mathrm{SL}}

\newcommand{\zz}{\mathfrak{z}}

\theoremstyle{plain}
\newtheorem{thm}{Theorem}[section]
\newtheorem{cor}[thm]{Corollary}
\newtheorem{lem}[thm]{Lemma}

\theoremstyle{definition}

\newtheorem*{rem}{Remark}

\numberwithin{equation}{section}

\newcommand{\pmat}[1]{\left( \smallmatrix #1 \endsmallmatrix \right)}

\def\lp{\left(}
\def\rp{\right)}

\def\a{\alpha}
\def\b{\beta}
\def\d{\delta}

\def\l{\lambda}

\def\t{\tau}

\def\del{  \partial}

\def\ddd{\mathrm{d}}

\newcommand{\re}{{\rm Re}}

\def\wh{\widehat}
\def\bar{\overline}
\newcommand{\andd}{\quad \mbox{ and } \quad}
\newcommand{\where}{\quad \mbox{ where }}

\def\slashchar#1{\setbox0=\hbox{$#1$}           
   \dimen0=\wd0                                 
   \setbox1=\hbox{/} \dimen1=\wd1               
   \ifdim\dimen0>\dimen1                        
      \rlap{\hbox to \dimen0{\hfil/\hfil}}      
      #1                                        
   \else                                        
      \rlap{\hbox to \dimen1{\hfil$#1$\hfil}}   
      /                                         
   \fi}                                        %

\setlist[itemize]{noitemsep, topsep=0pt}

\newcounter{exercise}
\renewcommand{\theexercise}{\thesection.\arabic{exercise}}

\mdfdefinestyle{exercise}{%
    linecolor=blue!40,
    outerlinewidth=2pt,
    leftline=false,rightline=false,
    skipabove=\baselineskip,
    skipbelow=\baselineskip,
    frametitle=\mbox{},
}
\newmdenv[%
    style=exercise,
    settings={\global\refstepcounter{exercise}},
    frametitlefont={\bfseries Exercise~\theexercise\quad},
]{exercise}
\newmdenv[%
    style=exercise,
    frametitlefont={\bfseries Exercise~\quad},
]{exercise*}

\newmdenv[%
    backgroundcolor=gray!8,
    linecolor=violet,
    outerlinewidth=1pt,
    roundcorner=3mm,
    skipabove=\baselineskip,
    skipbelow=\baselineskip,
]{boxes}
  
\makeatletter
\newcommand{\vast}{\bBigg@{2}}
\newcommand{\Vast}{\bBigg@{5}}
\makeatother

\renewcommand{\pmod}[1]{\  \,  \lp  \mathrm{mod} \,  #1 \rp}

\title[An exact formula]{An exact formula for $\mathbf{U (3)}$ Vafa-Witten invariants on $\IP^2$}

\author{Kathrin Bringmann}
\author{Caner Nazaroglu}
\address{Mathematical Institute, University of Cologne \\
Weyertal 86-90, 50931 Cologne, Germany}
\email{kbringma@math.uni-koeln.de}
\address{Mathematical Institute, University of Cologne \\
Weyertal 86-90,, 50931 Cologne, Germany}
\email{cnazarog@math.uni-koeln.de}

\keywords{}

\begin{document}

\begin{abstract}
Topologically twisted $\mathcal{N} = 4$ super Yang-Mills theory has a partition function that counts Euler numbers of instanton moduli spaces. On the manifold $\IP^2$ and with gauge group $\U (3)$ this partition function has a holomorphic anomaly which makes it a mock modular form of depth two. We employ the Circle Method to find a Rademacher expansion for the Fourier coefficients of this partition function. This is the first example of the use of Circle Method for a mock modular form of a higher depth.
\end{abstract}

\maketitle

\section{Introduction and statement of results}
Studying and understanding the structure of instanton moduli spaces is an interesting and important problem for both physics and mathematics. Although such spaces are quite intricate in general, one can go quite a long way in computing certain topological and analytic invariants. From a physical point of view, such invariants can be probed with topological field and string theories. This allows one to restrict attention to simpler and more tractable sectors of the original theory for which these moduli spaces are relevant. The concept of duality in physics then can lead to interesting mathematical relations between such invariants.

The particular example we focus on in this paper is the topological $\mathcal{N} = 4$ super Yang-Mills theory on a complex surface and with gauge group $\U (N)$ studied by Vafa and Witten \cite{Vafa:1994tf}. We call this topologically twisted theory {\it Vafa-Witten Theory}. 
Separating $\mathcal{Q}$-exact terms, the action grades configurations only by their instanton number. In this way, the partition function of Vafa-Witten theory contains a holomorphic $q$-series that counts (weighted) Euler numbers for instanton moduli spaces, which we denote by $f_{N,\mu} (\t)$, where $\mu$ is the magnetic t'Hooft flux and $\t \in \IH$, the complex upper half-plane, denotes the complexified gauge coupling.\footnote{We use the notation $f_{N,\mu} (\t)$ for the generating function of Vafa-Witten invariants and define the related function $h_{N,\mu}(\t)$ through $f_{N,\mu} (\t) \coloneqq \frac{h_{N,\mu} (\t)}{\eta (\t)^{3N}}$, where $\eta(\tau)$ is Dedekind's eta function. This notation is consistent with that of \cite{Bringmann:2010sd} but differs from that of \cite{Manschot:2017xcr}, where $h_{N,\mu} (\t) = \frac{f_{N,\mu} (\t)}{\eta (\t)^{3N}} $ is used to denote the generating function of Vafa-Witten invariants.
} The S-duality of $\mathcal{N}= 4$ Yang-Mills theory \cite{Montonen:1977sn, Osborn:1979tq, Witten:1978mh} then implies that such partition functions should be modular invariant yielding a nontrivial relation between the Euler numbers. In \cite{Vafa:1994tf}, this reasoning is applied as a test for the proposed duality by studying the partition functions for complex surfaces such as $K3$, $\mathrm{ALE}$ spaces, and $\IP^2$.

The relevant partition function for $\IP^2$ and with gauge group $\U (2)$ follows from the works of \cite{klyachko1991, yoshioka1994betti, yoshioka1995betti} and is expressed in terms of 
\begin{equation*}\label{eq:f2defn}
f_{2, \a} (\t) \coloneqq \frac{ h_{2,\a} (\t)}{\eta (\t)^6}, 
\qquad
\a \in \{ 0, 1 \},
\end{equation*}
where  
$\eta(\tau) \coloneqq q^{\frac{1}{24}}\prod_{n=1}^\infty(1-q^n)$ is \textit{Dedekind's eta-function}, $q := e^{2 \pi i \t}$,
and
\begin{equation*}
h_{2,\a} (\t) \coloneqq 3 h_\a (\t) \quad \mbox{where } \ 
h_\a (\t) \coloneqq \sum_{n=0}^\infty H(4n + 3 \a) \, q^{n + \frac{3\a}{4}},
\qquad
\a \in \{ 0, 1 \},
\end{equation*}
with $H(N)$ denoting the Hurwitz class numbers. The first few Fourier coefficients of   $ h_\a $ are given by
\begin{equation}\label{eq:h0Fourier}
h_0 (\t) = 
-\frac{1}{12} + \frac{1}{2} q + q^2 + \frac{4}{3} q^3 + \frac{3}{2} q^4 + 2 q^5 + 2 q^6 + 2 q^7 + 
 3 q^8 + \frac{5}{2} q^9 + 2 q^{10}  + O\left(q^{11}\right),
 \end{equation}
 \begin{equation}
\label{eq:h1Fourier}
h_1 (\t) =
\frac{1}{3} q^{\frac{3}{4}} + q^{\frac{7}{4}} + q^{\frac{11}{4}} + 2 q^{\frac{15}{4}} + q^{\frac{19}{4}} + 3 q^{\frac{23}{4}} + \frac{4}{3} q^{\frac{27}{4}} + 3 q^{\frac{31}{4}} + 2 q^{\frac{35}{4}}+ 4 q^{\frac{39}{4}}  + O\left(q^{\frac{43}{4}}\right).
\end{equation}
The function $h_\a $ is not modular invariant but one can add a piece that is non-holomorphic (and simpler) in a way that makes it modular invariant \cite{zagier75} (see equation \eqref{trans3} for the associated modular transformations). To be more precise, one defines
\begin{equation}\label{complete}
\wh{h}_\a (\t)  = \wh{h}_\a (\t, \bar{\t}) 
\coloneqq h_\a (\t) - \frac{i}{4 \sqrt{2} \pi}
\ \int\displaylimits_{-\bar{\t}}^{i \infty} \ 
\frac{\vartheta_{\frac{\a}{2}}  ( w)   }{\lp -i (w+\t) \rp^{\frac{3}{2}}}
\ddd w,
\end{equation}
where
\begin{equation*}
\vartheta_\ell(\tau) \coloneqq \sum_{n\in\ell+\IZ} q^{n^2}.
\end{equation*}
The function $f_{2, \a}$ is called a mixed mock modular form and is one of the first appearances of mock modular forms in physics. The theory of (mixed) mock modular forms has developed within the past two decades following the seminal work of Zwegers \cite{zwegers2008}.

The next obvious generalization is to $\U (3)$ Vafa-Witten theory on $\IP^2$ for which the relevant partition functions are
\begin{equation*}\label{eq:f3defn}
f_{3, \mu} (\t) \coloneqq \frac{h_{3,\mu} (\t)}{\eta (\t)^9}, 
\qquad
\mu \in \{ -1, 0, 1 \},
\end{equation*}
where the leading Fourier coefficients of $h_{3,\mu} $ are given by \cite{kool2009, Manschot:2010nc, Manschot:2014cca, Manschot:2017xcr, weist2009}
\begin{equation}
\label{eq:h30Fourier}
h_{3,0} (\t) = 
\frac{1}{9}  - q + 3 q^2 + 17 q^3 + 41 q^4 + 78 q^5 + 120 q^6 + 193 q^7 + 
 240 q^8 + 359 q^9 + 414 q^{10}+ O\left(q^{11}\right),
\end{equation}
\begin{equation}
\label{eq:h31Fourier}
h_{3,1} (\t) =h_{3,-1} (\t) = 
3 q^{\frac{5}{3}} + 15 q^{\frac{8}{3}} + 36 q^{\frac{11}{3}} + 69 q^{\frac{14}{3}} 
+ 114 q^{\frac{17}{3}} + 165 q^{\frac{20}{3}} + 246 q^{\frac{23}{3}} 
+ O\left(q^{\frac{26}{3}}\right).
\end{equation}
As in the case of $U(2)$, the function $h_{3, \mu}$ is not modular  but can be completed to a modular  object by adding an extra non-holomorphic term \cite{Manschot:2017xcr} (see equation \eqref{trans4} for the exact modular transformations) to define
\begin{equation}\label{complete2}
\wh{h}_{3,\mu} (\t, \bar{\t}) = h_{3,\mu} (\t) - 
\frac{9 \sqrt{3} i}{2 \sqrt{2} \pi}
\sum_{\a \pmod{2}}  \ 
\int\displaylimits_{-\bar{\t}}^{i \infty} \ 
\frac{\wh{h}_\a (\t, - w) \ \vartheta_{\frac{2\mu + 3\a}{6}}  (3 w)   }{\lp -i (w+\t) \rp^{\frac{3}{2}}}
d w,
\end{equation}
where for $\wh{h}_\a (\t, -w)$ we use equation \eqref{complete}, considering $\bar{\t}$ as an independent variable for which we then plug in $-w$. Because the holomorphic anomaly (i.e., the $\bar{\t}$ derivative) of the completion $\wh{h}_{3,\mu} (\t, \bar{\t})$ is given in terms of an ordinary mock modular form it is called a mock modular form of depth two according to the unpublished work of Zagier and Zwegers. The theory of such generalized mock modular forms at higher depth was developed recently in  \cite{Alexandrov:2016enp, funke2017theta,  Kudla2018, Nazaroglu:2016lmr, westerholt2016} via indefinite theta functions for lattices of arbitrary signature. These functions already found applications in physics \cite{Alexandrov:2016tnf, Alexandrov:2017qhn} and mathematics \cite{bringmann2016}. In fact, a key point in the analysis of \cite{Manschot:2017xcr} is the fact that $h_{3,\mu}$ can be written explicitly in terms of generalized Appell functions \cite{Manschot:2014cca} as in equations (6.10), (6.17), and (6.18) of \cite{Manschot:2017xcr} using which one can also find the Fourier expansion in \eqref{eq:h30Fourier} and \eqref{eq:h31Fourier}. Generalized Appell functions are particular examples of indefinite theta series. Using this fact, one can find the modular completion $\wh{h}_{3,\mu} (\t, \bar{\t})$, rewrite them in the form given in equation \eqref{complete2} and prove that they satisfy the modular transformations
\begin{equation*}
\widehat h_{3,\mu } (\t+1) 
= e^{-\frac{2\pi i \mu^2}{3} } \, \widehat h_{3,\mu} (\t), 
\qquad 
\widehat h_{3,\mu} \lp - \frac{1}{\t} \rp = 
\frac{(-i \t)^3}{\sqrt{3}} \sum_{\nu \pmod{3} } e^{-\frac{2 \pi i\mu \nu}{3} } \ 
\widehat h_{3,\nu } \lp \t \rp,
\end{equation*}
consistent with expectations from S-duality.

The goal of this paper is to exploit the modularity of $\U(3)$ Vafa-Witten invariants on $\IP^2$ to develop an exact formula for its Fourier coefficients, which makes its asymptotic form obvious with all the subleading terms calculable. For this purpose we use the Circle Method, which was first developed by Hardy and Ramanujan \cite{HardyRamanujan2, HardyRamanujan1} to study the asymptotic behavior of the (integer) partition function $p(n)$ and further refined by Rademacher \cite{Rademacher37} to give an exact formula for $p(n)$. We work with another version given by Rademacher \cite{Rademacher43} which is very suitable for understanding the origin of each term in such formulae. For the $\U (2)$ gauge group, this problem was considered in \cite{Bringmann:2010sd} in which the Circle Method was developed in order to deal with mixed mock modular forms. Our paper naturally extends this and uses the Circle Method for a higher depth mock modular form, taking as input only the form of modular transformations and completions and the leading Fourier coefficients of $h_{3,\mu}$.

We denote the $n$-th Fourier coefficient of $f_{3,\mu}$ by $\a_{3,\mu} (n)$. More specifically,
\begin{equation*}
f_{3,\mu} (\t) = \sum_{n=0}^\infty \a_{3,\mu} (n) \, q^{n- \Delta_\mu},
\where \ \ \ 
\Delta_0 \coloneqq \frac{3}{8}
\andd
\Delta_1 = \Delta_{-1} \coloneqq -\frac{31}{24}.
\end{equation*}
Our main theorem gives an exact formula for the Fourier coefficients, $\alpha_{3,\mu} (n)$.
To state it, we need some notation. Let
	$n_\mu:=n-\Delta_\mu$, $Q(x_1,x_2) \coloneqq x_1^2+x_2^2+x_1x_2$, and let
	$g^*_{k,r}$ and $g^*_{k,r_1,r_2}$ be given as
	\begin{equation*}
	g^*_{k,r}(w) \coloneqq w g_{\frac{r}{3k}}\left(\frac{3w}{2\sqrt{2}k}\right)\left(1-w^2\right)^{\frac54},\ \ \ 
	g^*_{k,r_1,r_2}(w_1, w_2):= g_{k,r_1,r_2}\left(\frac{3w_1}{2\sqrt{2}}, \frac{3w_2}{2\sqrt{2}}\right)\big(1-Q(w_1,w_2)\big)^{\frac54},
	\end{equation*}
	with the ingredients defined in equations \eqref{eq:gdefn}, \eqref{eq:fdefn}, \eqref{gnoto}, \eqref{eq:g_r1zero}, \eqref{eq:g_r2zero}, and \eqref{eq:g_r12zero}.
	Moreover the generalized Kloostermann sums are defined as
	\begin{equation*}
	K_k (\mu, \nu; n, r_1, r_2) \coloneqq 
	\sum_{\substack{ 0 \leq h < k \\ \gcd(h,k)=1}}   
	\zeta_{24k}^{-24 n_\mu h - (9+8Q(r_1,r_2))h'}
	\, \psi_{h,k}(\nu,\mu).
	\end{equation*}
	with multiplier system $\psi_{h,k}(\nu,\mu)$ given through equations \eqref{eq:multiplier}, \eqref{eq:etamultiplier}, \eqref{eq:multsys3}, and \eqref{eq:multsys3lambda} with $M = \pmat{ h' & -\frac{1+h h'}{k} \\ k & -h} \in \SL_2 (\IZ)$ for $h'$ satisfying $h h' \equiv -1 \pmod{k}$ and $\zeta_m \coloneqq e^{
		\frac{2 \pi i}{m}}$.
\begin{thm}\label{main theorem}
	We have
	\begin{align*}
	&\alpha_{3,\mu}(n)= \frac{\pi}{144} 
	\lp \frac{6}{n_\mu} \rp^{\frac{5}{4}}
	\sum_{k=1}^\infty \frac{K_k(\mu,0;n,0,0)}{k} I_{\frac52}\left(\frac{\pi\sqrt{6n_\mu}}{k}\right) \\
	&\quad -\frac{9\pi}{512}
	\lp \frac{6}{n_\mu} \rp^{\frac{5}{4}}
	 \sum_{\nu \pmod{3}} \sum_{k=1}^\infty \sum_{\substack{r \pmod{3k} \\ r\equiv \nu \pmod{3}}} \frac{K_k(\mu,\nu;n,r,0)}{k^2} \int_{-1}^1 g^*_{k,r}(w) I_{\frac52}\left(\frac{\pi\sqrt{6n_\mu\left(1-w^2\right)}}{k}\right)dw \\
	&\quad
	+ \frac{3\pi}{1024}
	\lp \frac{6}{n_\mu} \rp^{\frac{5}{4}}
	 \sum_{\nu \pmod{3}} \sum_{k=1}^\infty \sum_{\substack{r_1,r_2 \pmod{3k} \\ r_1\equiv r_2+\nu \pmod{3}}} \frac{K_k(\mu,\nu;n,r_1,r_2)}{k^3} \\
	&\hspace{4cm}\times\int_{Q(w_1,w_2)\leq 1} g^*_{k,r_1,r_2}(w_1,w_2) I_{\frac52}\left(\frac{\pi\sqrt{6n_\mu\big(1-Q(w_1,w_2)\big)}}{k}\right)dw_1 dw_2.
\end{align*}
\end{thm}

Using the asymptotic behavior of the Bessel functions we obtain the following.
\begin{cor}\label{ascor}
	We have, as $n\to \infty$,
	\begin{equation*}
	\alpha_{3,\mu}(n) = 
	\frac{1}{4(6n)^\frac32}e^{\pi\sqrt{6n}}
	\lp
	1 - \frac{81}{8 \pi (6 n)^\frac14} + \lp \frac{243\sqrt{3}}{16 \pi^2} - \frac{3}{\pi} \rp
	\frac{1}{(6 n)^\frac12} + O \lp n^{-\frac34} \rp
	\rp.
	\end{equation*}
\end{cor}
\begin{rem}
	One could determine further terms in the asymptotic expansion of $\alpha_{3,\mu}(n)$.
\end{rem}
The use of Circle Method to get an exact formula for Fourier coefficients of ordinary modular forms requires the precise transformation properties of these modular forms and their principal (or polar) parts which separate their growing behavior near the cusps. 
So for our case too, we start by reviewing modular transformation properties of $h_{3,\mu}$ and other associated functions that appear in its modular completion. For this purpose, in Section \ref{sec2}, we introduce certain multiplier systems that appear in these modular transformations and record some of their properties. Then, in Section \ref{sec3} we give the modularity behavior of the functions $f_{3,\mu}$ which lets us systematically work out the behavior of $f_{3,\mu}$ near the real line. Because of depth two mock modularity of  $f_{3,\mu}$, certain (one- and two-dimensional) theta integrals appear in the modular transformation equations. Next, in Section \ref{sec4}, we find Mordell-type representations for these theta integrals which reduce the $\tau$ dependence of the integrands to exponential functions. This allows us to split pieces that grow closer to the real line, which can be thought of as principal (or polar) parts of these contributions. In Section \ref{sec5}, we bound these integrals to find upper bounds on the error one gets by restricting to the these principal parts. Finally using these ingredients, in Section \ref{sec6}, we prove Theorem \ref{main theorem} using the Circle Method and find its asymptotics to prove Corollary \ref{ascor}. We finish the paper in Section \ref{sec7} by giving numerical results.

\section*{Acknowledgments} The research of the first author is supported by the Alfried Krupp Prize for Young University Teachers of the Krupp foundation
and the research leading to these results receives funding from the European Research Council under the European
Union's Seventh Framework Programme (FP/2007-2013) / ERC Grant agreement n. 335220 - AQSER.
The research of the second author is supported by the European Research Council under the European Union's Seventh Framework Programme (FP/2007-2013) / ERC Grant agreement n. 335220 - AQSER.
The authors thank Chris Jennings-Shaffer for helpful comments on an earlier version of the paper and thank the anonymous referees for their useful suggestions on the exposition of the paper.

\section{Multiplier systems}\label{sec2}
We start by introducing two multiplier systems, which we denote by $\psi_{2,M}$ and $\psi_{3,M}$ for 
$M = \left(\begin{smallmatrix} a & b \\ c & d  \end{smallmatrix}\right) \in\SL_2(\IZ)$. These arise as Weil representations associated with discriminant forms for $A_1$ and $A_2$ lattices, respectively. For easy reference, we give explicit formulae for both multiplier systems and refer the reader to \cite{CohenStromberg} for further details.

Firstly, we define $\psi_{2,M}$ as, with $\alpha,\beta\in\mathbb Z/2\mathbb Z$ ,
\begin{equation}\label{eq:multsys2}
\psi_{2,M} ( \a, \b) \coloneqq
\begin{cases}
i^{a b  \a^2}   e^{- \frac{\pi i}{4} (1 - \mathrm{sgn}(d))} \,  \d_{\a,\b}
\quad &\mbox{if } c=0, \\
\frac{e^{-\frac{\pi i}{4}   \mathrm{sgn} (c)}}{\sqrt{2 | c|}}
 \displaystyle\sum_{j=0}^{|c|-1} e^{\frac{\pi i}{2c} \lp  
	a \lp 2 j + \a \rp^2  - 2 \b \lp 2 j +\a \rp +d \b^2     \rp  }  
\quad &\mbox{if } c \neq 0,
\end{cases}
\end{equation}
where as usual $\delta_{\alpha,\beta}=0$ unless $\alpha=\beta$ in which case it equals 1.

Then, we define, with $\mu,\nu\in\mathbb Z/3\mathbb Z$,
\begin{equation}\label{eq:multsys3}
\psi_{3,M} ( \mu, \nu) \coloneqq 
\begin{cases}
e^{\frac{2\pi i}{3} a b \mu^2}   i^{\mathrm{sgn}(d)-1}  \d_{\mu,\nu}
\quad &\mbox{if } c=0, \\
\frac{i^{-\mathrm{sgn} (c)}}{\sqrt{3} | c| }
\l_{3,M} (\mu, \nu) 
\quad &\mbox{if } c \neq 0,
\end{cases}
\end{equation}
where
\begin{equation}\label{eq:multsys3lambda}
\l_{3,M} (\mu, \nu) 
\coloneqq \sum_{j_1,j_2=0}^{|c|-1} 
\exp \lp   \frac{2\pi i}{3c} \lp
a \mu^2 + d \nu^2 - 2 \mu \nu +3 a\left(j_1^2 -j_1 j_2 +j_2^2\right) + 3j_1(a\mu - \nu)
\rp  \rp.
\end{equation}
Importantly for our arguments, $\psi_{2,M}$ and $\psi_{3,M}$ are unitary. It is enough to verify this for the generators of $\text{SL}_2(\mathbb Z)$, $T \coloneqq \pmat{1 & 1 \\ 0 & 1}$ and $S \coloneqq \pmat{0 & -1 \\ 1 & 0}$. To state another property that is useful,  define $M^\sharp \coloneqq \pmat{a & -b \\ -c & d}$, where we assume from now on that $d>0$ if $c=0$. Then by directly inspecting equations \eqref{eq:multsys2}, \eqref{eq:multsys3}, and \eqref{eq:multsys3lambda}, one can see that
\begin{equation}\label{start}
\psi_{2,M^\sharp} (\a, \b) = \psi_{2,M}^* (\a,\b)
\andd
\psi_{3,M^\sharp} (\mu, \nu) = \psi_{3,M}^* (\mu,\nu),
\end{equation}
where $\ast$ denotes the complex conjugate.

Finally we give a lemma that states several (mock) modular transformations.
\begin{lem}
	We have, for $M=\left(\begin{smallmatrix}
	a & b \\ c & d
	\end{smallmatrix}\right)\in\SL_2(\IZ)$ and $\mathfrak z\in\mathbb C$ with $\textnormal{Im}(\mathfrak z)<0$,
	\begin{align}\label{trans1}
	\vartheta_{\frac{\a}{2}} \lp \frac{a \t + b}{c \t + d} \rp &= 
	\lp c \t + d \rp^{\frac{1}{2}} 
	\sum_{\b \pmod{2} } \psi_{2,M} ( \a, \b) \   \vartheta_{\frac{\b}{2}} \lp \t \rp ,\\
	\label{trans2}
	\vartheta_{\frac{2\mu + 3 \a}{6}} \lp 3 \frac{a \t + b}{c \t + d} \rp &= 
	\lp c \t + d \rp^{\frac{1}{2}} 
	\sum_{\nu   \pmod{3}} \sum_{\b \pmod{2} }
	\psi_{2,M}^* (\a,\b) \ \psi_{3,M} (\mu,\nu) \   \vartheta_{\frac{2\nu +3 \b}{6}} \lp 3\t \rp ,\\ 
	\label{trans3}
	\wh{h}_\a \lp \frac{a \t + b}{c \t + d},\frac{a\mathfrak z+b}{c\mathfrak z+d} \rp &= 
	\lp c \t + d \rp^{\frac{3}{2}} 
	\sum_{\b \pmod{2}} \psi_{2,M}^* ( \a, \b) \  \wh{h}_\b \lp \t,\mathfrak z \rp ,\\
	\label{trans4}
\wh{h}_{3,\mu} \lp \frac{a \t + b}{c \t + d} \rp &= 
\lp c \t + d \rp^3 
\sum_{\nu  \pmod{3} } \psi_{3,M}^* ( \mu, \nu)  \ \wh{h}_{3,\nu} \lp \t \rp.
\end{align}
\end{lem}
\begin{proof} 
	It is enough to show the claims for $M\in\{T,S\}$.\footnote{Note that the goal of this lemma is not to prove that $ \psi_{2,M} ( \a, \b)$ and $\psi_{3,M} ( \mu, \nu)$ are in fact multiplier systems for $\mathrm{SL}_2(\IZ)$. Instead, the aim is to show that these multiplier systems yield the multiplier systems of the (mock) modular forms we are interested in. Verifying that $T$ and $S$ transformations are consistent with the given multiplier systems is enough to show this claim.} For this purpose and as a reference, we list the relevant $T$ and $S$ transformations. In equations \eqref{trans1} and \eqref{trans2}, we have theta functions whose $T$ transformations immediately follow from their definition as $q$-series and whose $S$ transformations are proved in a standard way via Poisson summation and are well-known. More specifically, we have
\begin{align*}
\vartheta_{\frac{\a}{2} } (\t+1) 
&= i^{\a^2} \, \vartheta_{\frac{\a}{2}} (\t), 
&\quad 
\vartheta_{\frac{\a}{2} } \lp - \frac{1}{\t} \rp &= 
\frac{(-i \t)^{\frac{1}{2}}}{\sqrt{2}} \sum_{\b \pmod{2} } (-1)^{\a \b} \ 
\vartheta_{\frac{\b}{2} } \lp \t \rp,  
\\
\vartheta_{\frac{\ell}{6} } \lp 3( \t+1 ) \rp
&= e^{\frac{\pi i \ell^2}{6} } \, \vartheta_{\frac{\ell}{6}} (3 \t), 
&\quad 
\vartheta_{\frac{\ell}{6} } \lp - \frac{3}{\t} \rp &= 
\frac{(-i \t)^{\frac{1}{2}}}{\sqrt{6}} \sum_{r \pmod{6} } e^{- \frac{\pi i \ell  r}{3}  }\ 
\vartheta_{\frac{r}{6} } \lp 3 \t \rp .
\end{align*}
For equation \eqref{trans3}, we use the well-known transformation properties 
\begin{equation*}
\widehat h_{\a } (\t+1) 
= i^{-\a^2} \, \widehat h_{\a} (\t), 
\qquad 
\widehat h_{\a } \lp - \frac{1}{\t} \rp = -
\frac{(-i \t)^{\frac{3}{2}}}{\sqrt{2}} \sum_{\b \pmod{2} } (-1)^{\a \b} \ 
\widehat h_{\b } \lp \t \rp .
\end{equation*}
Finally, we compare equation \eqref{trans4} to the transformation properties given in \cite{Manschot:2017xcr}
\begin{equation*}
\widehat h_{3,\mu } (\t+1) 
= e^{-\frac{2\pi i \mu^2}{3} } \, \widehat h_{3,\mu} (\t), 
\qquad 
\widehat h_{3,\mu} \lp - \frac{1}{\t} \rp = 
\frac{(-i \t)^3}{\sqrt{3}} \sum_{\nu \pmod{3} } e^{-\frac{2 \pi i\mu \nu}{3} } \ 
\widehat h_{3,\nu } \lp \t \rp .
\end{equation*}
\end{proof}

We also need the modular transformations for the Dedekind $\eta$-function:
\begin{equation}\label{etatrans}
\eta\left(\frac{a\t+b}{c\t+d}\right)=\psi\begin{pmatrix}
a & b\\c & d
\end{pmatrix}
\,
(-i(c\t+d))^{\frac{1}{2}}
\  \eta(\t),
\end{equation}
where for $c \neq 0$, we define
\begin{equation}\label{eq:etamultiplier}
\psi\begin{pmatrix}
a & b\\c & d
\end{pmatrix}:=\begin{cases}
\left(\frac{d}{|c|}\right)e^{\frac{\pi i}{12}\left((a+d)c-bd\left(c^2-1\right)-3c+3\right)}\quad &\text{ if c is odd,}\\
\left(\frac{c}{d}\right)e^{\frac{\pi i}{12}\left(ac\left(1-d^2\right)+d(b-c+3)\right)} &\text{ if c is even}.
\end{cases}
\end{equation}

Lastly, using equations \eqref{eq:multsys3}, \eqref{eq:multsys3lambda}, and \eqref{eq:etamultiplier} we set
\begin{equation}\label{eq:multiplier}
\chi_M (\mu, \nu) \coloneqq i \, \psi(M)^9 \, \psi_{3,M} (\mu,\nu).
\end{equation}

\section{The transformation behavior of the functions $f_{3,\mu}$}\label{sec3}
For $j,\nu\in\IN_0, \varrho\in\IQ$, define the theta integrals
\begin{align}\label{defineTint}
\mathcal{E}_{1, j, \varrho} (\t) &\coloneqq 
\int\displaylimits_{\varrho}^{i \infty} \, 
\frac{ \vartheta_{\frac{j}{6}} (3 w)  }{\lp -i \lp w + \t \rp \rp^{\frac{3}{2}}  }
d w,
\\  \label{defineTint2}
\mathcal{E}_{2, \nu,\varrho } (\t)
&\coloneqq 
\sum_{\a     \pmod{2}}
\int\displaylimits_{\varrho }^{i \infty} \, 
\int\displaylimits_{w_1 }^{i \infty} \, 
\frac{\vartheta_{\frac{\a}{2}} (w_2)  \vartheta_{\frac{2 \nu + 3\a}{6}} (3 w_1) }
{\lp -i \lp w_2 + \t \rp \rp^{\frac{3}{2}}\lp -i \lp w_1 + \t \rp \rp^{\frac{3}{2}}} 
d w_2 \, d w_1.
\end{align}

The following lemma finds the mock modular transformation of $f_{3,\mu}$.
\begin{lem}\label{lem:mocktrans}
	For $M=(\begin{smallmatrix}
	a & b\\ c & d
	\end{smallmatrix})\in\textnormal{SL}_2(\mathbb Z)$ with $c \neq 0$, $f_{3,\mu} (\t)(-i(c\t+d))^{-\frac32}$ equals
\begin{align*}
&\sum_{\nu     \pmod{3}} \chi_M (\nu, \mu)
\Bigg(
f_{3,\nu} \lp \frac{a \t + b}{c \t + d} \rp
- \frac{9\sqrt{3}i}{2\sqrt{2}\pi} 
\sum_{\a     \pmod{2}}
f_\alpha \lp \frac{a \t + b}{c \t + d} \rp
\mathcal{E}_{1,2\nu+3\a,-\frac{a}{c}} \lp \frac{a \t + b}{c \t + d} \rp
\\
&\hspace{5cm}- \frac{9 \sqrt{3}}{16\pi^2}
f \lp \frac{a \t + b}{c \t + d} \rp
\mathcal{E}_{2,\nu,-\frac{a}{c}} \lp \frac{a \t + b}{c \t + d} \rp
\Bigg)  ,
\notag
\end{align*}
where
\begin{equation*}
f_\alpha (\t) \coloneqq \frac{h_\a (\t)}{\eta(\t)^9}, \qquad
f (\t) \coloneqq \frac{1}{\eta(\t)^9}.
\end{equation*}
\end{lem}

\begin{proof}
	Using \eqref{complete2}, \eqref{trans4}, and the unitarity of $\psi_{3,M}$, we find that
	\begin{align}
	&h_{3,\mu} (\t) 
	- (c \t + d)^{-3} \sum_{\nu   \pmod{3}}
	\psi_{3,M} (\nu,\mu) \, 
	h_{3,\nu} \lp \frac{a \t + b}{c \t + d} \rp  
	\notag
	\\
	\label{eq:h3_error}
	&=\frac{9\sqrt{3}i}{2\sqrt{2}\pi}
	\sum_{\a    \pmod{2}}
	\ \left( \ 
	\int\displaylimits_{-\bar{\t}}^{i \infty} \, 
	\frac{ \wh{h}_\a (\t, -w) \, \vartheta_{\frac{2 \mu + 3\a}{6}} (3 w)  }{\lp -i \lp w + \t \rp \rp^{\frac{3}{2}}  }
	d w  \right. 
	 \\
	&\left.\hspace{4cm}
	- (c \t + d)^{-3} 
	\sum_{\nu   \pmod{3}}
	\psi_{3,M} (\nu,\mu)
	\int\displaylimits_{-\frac{a \bar{\t} + b}{c \bar{\t} + d} }^{i \infty} \, 
	\frac{ \wh{h}_\a \lp  \frac{a \t + b}{c \t + d}, -w \rp \, \vartheta_{\frac{2 \nu + 3\a}{6}} (3 w)  }{\lp -i \lp w + \frac{a \t + b}{c \t + d} \rp \rp^{\frac{3}{2}}  }
	d w
	\right).
	\notag
	\end{align}
	
	To simplify, we rewrite the
	first term on the right-hand side of \eqref{eq:h3_error}. For this, we make the change of variables $w\mapsto \frac{dw+b}{cw+a}$ and use the unitary of the multipliers, \eqref{start}, \eqref{trans2}, and \eqref{trans3},
to obtain that
	\begin{align*}
	\widehat{h}_\alpha \left(\t,-\frac{dw+b}{cw+a}\right)&=(c\t+d)^{-\frac32}\sum_{\beta\pmod{2}}\psi_{2,M}(\beta,\alpha)\widehat h_\beta\left(\frac{a\t+b}{c\t+d},-w\right),\\
	\vartheta_{\frac{2 \mu + 3\a}{6}}  (3 \t) 
&=
	(- c \t + d)^{-\frac{1}{2}} 
	\sum_{\nu     \pmod{3}}  
	\sum_{\b     \pmod{2}}
	\psi_{2,M}^* (\b, \a) \,  \psi_{3,M} (\nu, \mu) \,
	\vartheta_{\frac{2 \nu + 3\b}{6}}  \lp 3 \frac{a \t - b}{-c \t + d} \rp .
	\end{align*}
	Plugging these in and simplifying, we find that \eqref{eq:h3_error} equals 
	\begin{equation*}
	- \frac{9\sqrt{3} i}{2\sqrt{2} \pi} (c \t + d)^{-3} 
	\sum_{\nu    \pmod{3}}\psi_{3,M} (\nu,\mu) \, 
	\sum_{\a    \pmod{2}} \ 
	\int\displaylimits_{-\frac{a}{c} }^{i \infty} \, 
	\frac{ \wh{h}_\a \lp  \frac{a \t + b}{c \t + d}, -w \rp \, \vartheta_{\frac{2 \nu + 3\a}{6}} (3 w)  }{\lp -i \lp w + \frac{a \t + b}{c \t + d} \rp \rp^{\frac{3}{2}}  }
	d w,
	\end{equation*}
	Using  \eqref{complete} and \eqref{etatrans} then finishes the claim.
\end{proof}

\section{Eichler integrals}\label{sec4}

In this section, we rewrite the theta integrals, defined in \eqref{defineTint} and \eqref{defineTint2}, as Eichler integrals. Throughout the section, we assume that $\mathrm{Re} (z) >0$ and $h',k \in \IZ$ with $k>0$.
\subsection{The one-dimensional case}

\begin{lem}\label{lem3.1}
	We have
	\begin{equation*}
	\mathcal E_{1,j,-\frac{h'}{k}}\left(\frac{h'}{k}+iz\right) =  \frac{\pi i}{3\sqrt{6}k} \sum_{\substack{r\pmod{6k} \\ r\equiv j\pmod{6}}} \zeta_{12k}^{-h'r^2} \int_{\IR} w g_{\frac{r}{6k}}\left(\frac{w}{2k}\right)e^{-\frac{1}{6} \pi z w^2}dw,
	\end{equation*}
	where, for $c\in\mathbb Q$ and $w\in\mathbb C$, 
\begin{equation}\label{eq:gdefn}
	g_c(w) \coloneqq \frac{\sinh\left(\frac{2\pi w}{3}\right)}{\cosh\left(\frac{2\pi w}{3}\right)-\cos(2\pi c)}.
\end{equation}
\end{lem}

\begin{proof}
	Lemma \ref{lem3.1} is well-known to experts, however, for the convenience of the reader, we give a proof.
	Plugging in definition \eqref{defineTint}, we rewrite 
	\begin{align*}
	\mathcal E_{1,j,-\frac{h'}{k}}\left(\frac{h'}{k}+iz\right)= \int_0^{i\infty} \frac{\vartheta_{\frac{j}{6}}\left(3\left(w-\frac{h'}{k}\right)\right)}{\left(-i\left(iz+w\right)\right)^{\frac32}}dw.
	\end{align*}
	We next assume that $z>0$ and argue via analytic continuation.
	Letting $w=it$, using
	the identity
	\begin{align*}
	\int_{\mathbb R}e^{-2\pi w^2v}\frac{w}{w-is}dw&=\frac{1}{2\sqrt{2}}\int_0^\infty\frac{e^{-2\pi ts^2}}{(t+v)^\frac32}dt,
	\end{align*}
	and inserting the Fourier expansion of $\vartheta_{\frac{j}{6}}$,
	we obtain 
	\begin{align*}
	\mathcal E_{1,j,-\frac{h'}{k}} \left(\frac{h'}{k}+iz\right) &=\frac{\sqrt{2} i}{\sqrt{3}} \sum_{\substack{ r\pmod{6k} \\ r\equiv j \pmod{6}}} \zeta_{12k}^{-h'r^2} \sum_{\substack{ m\in \IZ \\ m\equiv r\pmod{6k}}} \int_{\IR} \frac{we^{-\frac{1}{6} \pi z w^2}}{w-im} dw.
	\end{align*}
	Using
	\begin{equation}\label{cotid}
	\pi \cot(\pi x) = \lim_{M\to \infty}\sum_{m=-M}^M\frac{1}{x+m},
	\end{equation}
	we may then show that
	\begin{align*}
	\mathcal E_{1,j,-\frac{h'}{k}}\left(\frac{h'}{k}+iz\right) = -\frac{\pi}{3\sqrt{6}k} \sideset{}{^{}}\sum_{\substack{r\pmod{6k} \\ r\equiv j \pmod{6}}} \zeta_{12k}^{-h'r^2} \int_\IR we^{-\frac{1}{6}\pi z w^2} \cot\left(\pi\frac{iw+r}{6k}\right)dw.
	\end{align*}
	The claim of the Lemma follows, using
	\begin{equation}\label{cotid2}
	\cot(x+iy) = \frac{\sin(2x)}{\cosh(2y)-\cos(2x)} - i \frac{\sinh(2y)}{\cosh(2y)-\cos(2x)}
	\end{equation}
	and the fact that
	the contribution of the first term vanishes.
\end{proof}

\subsection{The two-dimensional case}
The main goal of this section is to write the two-dimensional Eichler integral as a Mordell integral. Such integrals were first found by Kaszian, Milas, and the first author in \cite{BKM}.
To state the main result,
we define the function $g_{k, \bf{r}} : \IR^2 \to \IR$ as follows. Set 
\begin{equation}\label{eq:fdefn}
f_c (w) := \frac{\sin(2\pi c)}{\cosh \left(\frac{2\pi w}{3}\right) -\cos(2\pi c)},
\end{equation}
and write here and throughout this paper, vectors as $\boldsymbol{z}=:(z_1,z_2)$.

	If $r_1,r_2\not\equiv 0 \pmod{3k}$, then we define 
	\begin{equation}\label{gnoto}
	g_{k,\boldsymbol{r}} (\boldsymbol{w}) \coloneqq 
	\left(w_1^2+w_2^2+4w_1w_2\right) 
	\left(g_{\frac{r_1}{3k}}\left(\frac{w_1}{k}\right)g_{\frac{r_2}{3k}}\left(\frac{w_2}{k}\right)-f_{\frac{r_1}{3k}}\left(\frac{w_1}{k}\right)f_{\frac{r_2}{3k}}\left(\frac{w_2}{k}\right)\right).
	\end{equation}
	
	If $r_1 \equiv 0, r_2 \not\equiv 0 \pmod{3k}$, then we let
	\begin{align}\label{eq:g_r1zero}
	g_{k,(0,r_2)} (\boldsymbol{w}) \coloneqq \left(w_1^2+w_2^2+4w_1w_2\right)g_0\left(\frac{w_1}{k}\right)g_{\frac{r_2}{3k}}\left(\frac{w_2}{k}\right)-\frac{3k}{\pi w_1}\left(w_2+\frac{w_1}{2}\right)^2g_{\frac{r_2}{3k}}\left(\frac{w_2+\frac{w_1}{2}}{k}\right).
	\end{align}
	
	If $r_1 \not\equiv 0, r_2 \equiv 0 \pmod{3k}$, then set 
	\begin{equation}\label{eq:g_r2zero}
	g_{k,(r_1,0)}  \coloneqq g_{k,(0,r_1)} , 
	\end{equation}
	
	Finally, if $r_1,r_2 \equiv 0  \pmod{3k}$, then 
	\begin{align}\label{eq:g_r12zero}
	g_{k,\boldsymbol{0}} (\boldsymbol{w}) \coloneqq \left(w_1^2+w_2^2+4w_1w_2\right)g_0\left(\frac{w_1}{k}\right)g_{0}\left(\frac{w_2}{k}\right)-\frac{3k}{\pi w_1}\left(w_2+\frac{w_1}{2}\right)^2g_{0}\left(\frac{w_2+\frac{w_1}{2}}{k}\right)\\\notag-\frac{3k}{\pi w_2}\left(w_1+\frac{w_2}{2}\right)^2g_{0}\left(\frac{w_1+\frac{w_2}{2}}{k}\right).
	\end{align}

\begin{thm}\label{thm2dim}
	We have, with $\boldsymbol{dw}:=dw_1dw_2$
	\begin{align*}
	&\mathcal{E}_{2,\nu,-\frac{h'}{k}} \lp \frac{h'}{k}  + i z \rp
	= -\frac{2\pi^2}{27 \sqrt{3}k^2} 
	\quad
	\sum_{\substack{ r_1, r_2 \pmod{3k} \\ r_1 \equiv r_2 + \nu \pmod{3} }}
	\zeta_{3k}^{-h'Q(\boldsymbol{r})}
	\int\displaylimits_{\IR^2}
	g_{k,\boldsymbol{r}} (\boldsymbol{w}) 
	e^{-\frac{2}{3} \pi z Q(\boldsymbol{w})}
	\boldsymbol{dw}.
	\end{align*}
\end{thm}

Before proving Theorem \ref{thm2dim}, we require an auxiliary lemma.
For this, we introduce two involutions $\iota_1$ and $\iota_2$ acting on $\mathbb R^2$ that leave the quadratic form $Q^*(x_1,x_2) \coloneqq x_1^2  +x_2^2 - x_1 x_2$ for $x_1, x_2 \in \IR$ invariant, namely
\begin{equation*}
\iota_1:(x_1, x_2) \mapsto (-x_2,-x_1)
\andd
\iota_2:(x_1, x_2) \mapsto (x_2-x_1,x_2).
\end{equation*}
These two involutions are the generators of the Weyl group for the root lattice $A_2$, which is isomorphic to the symmetric group $S_3$. We average a function $h: \IR^2 \to \IC$ over the orbit of a point $(x_1, x_2)$ under the group generated by the involutions $\iota_1$ and $\iota_2$, namely
\begin{align*}
(x_1, x_2) 
&\xmapsto{\iota_1} (-x_2, -x_1)
\xmapsto{\iota_2} (x_2-x_1, -x_1)
\xmapsto{\iota_1} (x_1, x_1-x_2)
\xmapsto{\iota_2} (-x_2, x_1-x_2)\\
&\xmapsto{\iota_1} (x_2-x_1, x_2)
\xmapsto{\iota_2} (x_1, x_2).
\end{align*}
We then define the average of $h$ as
\begin{align*}
\sideset{}{^a}\sum_{\boldsymbol{x}} h(\boldsymbol{x}) \coloneqq
\frac{1}{6}
\Big(
&h(x_1, x_2)  + h(-x_2, -x_1) + h(x_2-x_1, -x_1) 
\\
&\quad
+ h(x_1, x_1-x_2)   + h(-x_2, x_1-x_2) + h(x_2-x_1, x_2)
\Big).
\notag
\end{align*}
Note that
\begin{equation*}\label{av2}
\sideset{}{^a}\sum_{\boldsymbol{x}} h(\boldsymbol{x}) =
\frac{1}{2}
\sideset{}{^a}\sum_{\boldsymbol{x}}\lp h(x_1,x_2) + h(-x_2, -x_1) \rp
=
\frac{1}{2}
\sideset{}{^a}\sum_{\boldsymbol{x}} \lp h(x_1,x_2) + h(x_2-x_1, x_2) \rp.
\end{equation*}
We also define (excluding $2x_2-x_1,2x_1-x_2=0$)
\begin{align*}
F(\boldsymbol{x})=F(\boldsymbol{x};z) &\coloneqq \int\displaylimits_{\IR^2} 
\frac{(w_1 + w_2)^2 e^{-\frac{2}{3}\pi z Q(\boldsymbol{w}) } }
{\lp w_1 - i (2x_2 -x_1) \rp \lp w_2 - i (2x_1 -x_2) \rp} \boldsymbol{dw},\\
G(\boldsymbol{x})=G(\boldsymbol{x};z) &\coloneqq \frac{\sqrt{3}}{2} 
\int\displaylimits_0^\infty
\frac{e^{-\frac{3}{2} \pi x_2^2 w_1 }}{\lp w_1 + z \rp^{\frac{3}{2}}} 
\int\displaylimits_{w_1}^\infty
\frac{e^{-\frac{1}{2}\pi (2x_1-x_2)^2 w_2 }}{\lp w_2 +z \rp^{\frac{3}{2}}}  
d w_2 dw_1.
\end{align*} 
These functions agree when averaged.

\begin{lem}\label{lem:integral_identity}
	Let $x_1, x_2 \in \IR$ with $2x_2 - x_1 \neq 0$, $2x_1 - x_2 \neq 0$,  and
	$x_1 + x_2 \neq 0$. Then we have
	\begin{equation}\label{FG}
	\sideset{}{^a}\sum_{\boldsymbol{x}} F(\boldsymbol{x}) 
	=
	\sideset{}{^a}\sum_{\boldsymbol{x}} G(\boldsymbol{x}) .
	\end{equation}
	We have, as $x_1^2+x_2^2\to\infty$, 
	\begin{equation}\label{boundsG}
	G(\boldsymbol{x}) \ll \begin{cases}
	x_2^{-2}(2x_1-x_2)^{-2}\quad&\textnormal{ if } x_2,2x_1-x_2\neq 0,\\
	x_2^{-2}\quad &\textnormal{ if } 2x_1-x_2=0, x_2\neq 0,\\
	(2x_1-x_2)^{-2}\quad&\textnormal{ if } x_2=0, x_1\neq 0.
	\end{cases}
	\end{equation}
\end{lem}

\begin{proof}
	The bounds in \eqref{boundsG} are direct, thus we only prove \eqref{FG}.
	Via analytic continuation, it is enough to show this identity for $z\in\mathbb R^+$, which we assume from now on. We first claim that
	\begin{align}\label{averageF}
	\sideset{}{^a}\sum_{\boldsymbol{x}} F(\boldsymbol{x}) 
	=&
	\frac{\sqrt{3}}{z} \sideset{}{^a}\sum_{\boldsymbol{x}}\left[
	\frac{\del^2}{\del \mathfrak{z}^2}  \Bigg(
	(x_2 + \mathfrak{z}) (2 x_1 - x_2 +\mathfrak{z})e^{2 \pi z Q^*(\boldsymbol{x})+2 \pi z  (x_1 + x_2)\mathfrak{z}}\right.
	\\
	&\qquad\qquad\qquad\times\left.
	\int\displaylimits_1^\infty
	e^{- \frac{3}{2}\pi z (x_2 + \mathfrak{z})^2 w_1^2}
	\int\displaylimits_{w_1}^\infty
	e^{- \frac{1}{2} \pi z (2 x_1 - x_2 +\mathfrak{z})^2 w_2^2 } 
	d w_2
	d w_1
	\Bigg)\right]_{\mathfrak{z}=0}.
	\notag
	\end{align} 
	To prove \eqref{averageF}, we use the change of variables $w_1 \mapsto \frac{w_1 - w_2}{2}$ in the definition of $F(\boldsymbol{x})$, to rewrite
	\begin{align*}
	F(\boldsymbol{x})  &= -
	\frac{1}{4\pi^2 z^2} 
	e^{2\pi zQ^*(\boldsymbol{x})}\left[\frac{\partial^2}{\partial \zz^2}\left(e^{2\pi z(x_1+x_2)\zz} \mathcal F_{\boldsymbol{x}}(\zz,1)\right)\right]_{\zz=0},
	\end{align*}
	where
	\begin{align*}
	\mathcal F_{\boldsymbol{x}}(\zz,t) :=
	\int\displaylimits_{\IR^2 - i \, \pmat{3 x_2 \\ 2 x_1 - x_2}} 
	\frac{e^{-\frac{1}{6} \pi z w_1^2 -\frac{1}{2} \pi z w_2^2- \pi i z t \lp (x_2 +\mathfrak{z}) (w_1-w_2) + 2 (x_1 +\mathfrak{z}) w_2  \rp}}
	{\lp w_1 - w_2 \rp  w_2 } \boldsymbol{dw}.
	\end{align*}
	Define
	\[
	\mathcal H(t):=\left[\frac{\partial^2}{\partial \mathfrak z^2}\sideset{}{^a}\sum_{\boldsymbol{x}} e^{2\pi zQ^*(\boldsymbol{x})+2\pi z (x_1+x_2)\mathfrak{z}}\mathcal F_{\boldsymbol{x}}(\zz,t)\right]_{\mathfrak z=0}.
	\]	
	We identify $\mathcal H(t)$ by determining its derivative and its limiting behavior.
	We first compute
	\begin{align}
	&\frac{\partial}{\partial t}\mathcal F_{\boldsymbol{x}}(\zz,t)=
	- \pi i z (x_2 + \mathfrak{z})
	\int\displaylimits_{\IR^2 -i \, \pmat{3 x_2 \\ 2 x_1 - x_2} }  
	e^{-\frac{1}{6}\pi z w_1^2 -\frac{1}{2}\pi z w_2^2- \pi i z t \lp (x_2 +\mathfrak{z}) (w_1-w_2) + 2 (x_1 +\mathfrak{z}) w_2  \rp} 
	\frac{1}{w_2}
	\boldsymbol{dw} 
	\notag
	\\
	& \ \ 
	- 2 \pi i z (x_1 + \mathfrak{z})
	\int\displaylimits_{\IR^2 -i \, \pmat{3 x_2 \\ 2 x_1 - x_2}}  
	e^{-\frac{1}{6} \pi z w_1^2 -\frac{1}{2} \pi z w_2^2- \pi i z t \lp (x_2 +\mathfrak{z}) (w_1-w_2) + 2 (x_1 +\mathfrak{z}) w_2  \rp} 
	\frac{1}{w_1-w_2}
	\boldsymbol{dw} .
	\label{tderint}
	\end{align}
	Evaluating the integral in $w_1$ as Gaussian, one can show that the first term in \eqref{tderint} equals
	\begin{align}\label{first}
	- \pi i \sqrt{6 z} (x_2 + \mathfrak{z})  e^{- \frac{3}{2} \pi z(x_2 + \mathfrak{z})^2 t^2} \mathcal G_{\boldsymbol{x}}(\zz,t),
	\end{align}
	where 
	\[
	\mathcal G_{\boldsymbol{x}}(\zz,t):=\qquad\int\displaylimits_{\mathclap{\IR - i (2 x_1 - x_2)}}  
	\quad \ e^{ -\frac{1}{2} \pi z w_2^2- \pi i z t (2 x_1 - x_2 +\mathfrak{z})w_2 } 
	\frac{d w_2}{w_2}.
	\]
	
For the second term on the right-hand side of \eqref{tderint} we change variables and take a Gaussian integral to show that it equals
	\begin{equation}\label{GaussIntEvasecond}
	- \pi i \sqrt{6 z} (x_1 + \mathfrak{z})  e^{- \frac{3}{2} \pi z (x_1 + \mathfrak{z})^2 t^2}
	\qquad\int\displaylimits_{\mathclap{\IR - i (2 x_2 - x_1)}}  \quad \
	e^{ -\frac{1}{2} \pi z w_2^2- \pi i z t  (2 x_2 - x_1 +\mathfrak{z}) w_2} 
	\frac{d w_2}{w_2} .
	\end{equation}
	Note that \eqref{first} and \eqref{GaussIntEvasecond} are mapped to each other when applying the involution $\iota_1$ and changing $\mathfrak z$ into $-\mathfrak z$ and $w_2$ into $-w_{2}$; note that the prefactor $e^{2 \pi z Q^*(\boldsymbol{x})+2 \pi z \mathfrak{z} (x_1 + x_2)}$ is invariant under these exchanges. 
	Thus we obtain
	\begin{equation}\label{idHp}
	\mathcal H'(t)=-2\pi i\sqrt{6z}\left[\frac{\partial^2}{\partial \mathfrak z^2}\sideset{}{^a}\sum_{\boldsymbol x}(x_2+\mathfrak z)e^{-\frac{3}{2} \pi z(x_2+\mathfrak z)^2t^2+2\pi zQ^*(\boldsymbol{x})+2\pi z(x_1+x_2)\zz}\mathcal G_{\boldsymbol{x}}(\zz,t)\right]_{\mathfrak z=0}.
	\end{equation}
	It is not hard to show that
	\[
	\mathcal G_{\boldsymbol{x}}(\zz,t)=\pi i (2x_1-x_2+\mathfrak z)\sqrt{2z}\int_t^\infty e^{-\frac{1}{2} \pi z(2x_1-x_2+\mathfrak z)^2w^2}dw.
	\]
	Plugging this into \eqref{idHp} gives that
	\begin{multline*}
	\mathcal H'(t)=4\sqrt{3}\pi^2z\left[\frac{\partial^2}{\partial\mathfrak z^2}\sideset{}{^a}\sum_{\boldsymbol{x}}(x_2+\mathfrak z)(2x_1-x_2+\mathfrak z)e^{-\frac{3}{2} \pi z (x_2+\mathfrak z)^2t^2+2\pi zQ^*(\boldsymbol{x})+2\pi z(x_1+x_2)\mathfrak z}\right.\\
	\hspace{4cm} \left. \times\int_t^\infty e^{-\frac{1}{2} \pi z (2x_1-x_2+\mathfrak z)^2w^2}dw\right]_{\mathfrak z=0}.
	\end{multline*}
	Using that
	$\lim_{t\rightarrow\infty}\mathcal H(t)=0$, we then obtain
	\begin{multline*}
	\mathcal H(t)=-4\sqrt{3}\pi^2z\left[\frac{\partial^2}{\partial\mathfrak z^2}\sideset{}{^a}\sum_{\boldsymbol{x}}(x_2+\mathfrak z)(2x_1-x_2+\mathfrak z)e^{2\pi zQ^*(\boldsymbol{x})+2\pi z(x_1+x_2)\zz}\right.\\
	\left.\times\int_t^\infty e^{-\frac{3}{2}\pi z(x_2+\mathfrak z)^2w_1^2}\int_{w_1}^\infty e^{-\frac{1}{2} \pi z(2x_1-x_2+\mathfrak z)^2w_2^2}dw_2dw_1\right]_{\mathfrak z=0}.
	\end{multline*}
	Plugging in gives \eqref{averageF}. 
	
	Setting
	\begin{equation*}
	F_0(\mathfrak z) \coloneqq 
	(x_2+\mathfrak z)(2x_1-x_2+\mathfrak z)e^{2\pi z(x_1+x_2)\zz}\int_1^\infty e^{-\frac{3}{2} \pi z(x_2+\mathfrak z)^2 w_1^2}\int_{w_1}^\infty e^{-\frac{1}{2} \pi z (2x_1-x_2+\mathfrak z)^2 w_2^2}dw_2dw_1,
	\end{equation*} 
	a direct calculation shows that
	\begin{align*}
	&F''_0(0)=4\pi^2z^2(x_1+x_2)^2x_2(2x_1-x_2)\int_1^\infty e^{-\frac{3}{2} \pi z x_2^2 w_1^2}\int_{w_1}^\infty e^{-\frac{1}{2} \pi z (2x_1-x_2)w_2^2}dw_2dw_1\\
	&  
	-\pi z (2x_1-x_2)(4x_1+x_2)e^{-\frac{3}{2}\pi z x_2^2}
	\int\displaylimits_{1}^\infty e^{-\frac{1}{2} \pi z (2x_1-x_2)^2w_2^2}dw_2+e^{-2\pi zQ^*(\boldsymbol{x})}-\frac{x_1^2-x_2^2}{2\pi zQ^*(\boldsymbol{x})^2}e^{-2\pi zQ^*(\boldsymbol{x})}.
	\end{align*}
	Noting that $\displaystyle\sideset{}{^a}\sum_{\boldsymbol{x}} (x_2^2 - x_1^2)=0$ and
	\begin{align*}
	\Big[(x_1 + x_2)^2 x_2 (2 x_1 - x_2) \Big]_{x_1 \mapsto x_2 - x_1}
	+(x_1 + x_2)^2 x_2 (2 x_1 - x_2) &= 3 x_2^2 (2 x_1 - x_2)^2,\\
	\left[(x_2 - 2 x_1)(x_2 + 4x_1) \right]_{x_1 \mapsto x_2 - x_1}
	+(x_2 - 2 x_1)(x_2 + 4x_1)&= -4 (2x_1 - x_2)^2,
	\end{align*}
	yields, after a change of variables,
	\begin{align}\label{finalavgsum}
	\sideset{}{^a}\sum_{\boldsymbol{x}} F(\boldsymbol{x})&=
	\frac{\pi \sqrt{3}}{2} \sideset{}{^a}\sum_{\boldsymbol{x}} \Bigg(
	3 \pi z  x_2^2  (2 x_1 - x_2)^2
	\int\displaylimits_0^\infty
	\frac{e^{- \frac{3}{2} \pi z  x_2^2 w_1}}{\sqrt{w_1 +1}}
	\int\displaylimits_{w_1}^\infty
	\frac{e^{- \frac{1}{2} \pi z  (2 x_1 - x_2)^2 w_2 } }{\sqrt{w_2 + 1}}
	d w_2
	d w_1 
	\\\notag
	&\hspace{6.3cm}
	-2  (2x_1 - x_2)^2
	\int\displaylimits_0^\infty
	\frac{e^{- \frac{1}{2} \pi z  (2 x_1 - x_2)^2 w }}{\sqrt{w+1}}
	d w
	+ \frac{2}{\pi z}
	\Bigg).
	\notag 
	\end{align}
	Using integration by parts twice, we obtain
	that
	\eqref{finalavgsum} equals
	\begin{align*}
	&\frac{\pi \sqrt{3}}{2}\sideset{}{^a}\sum_{\boldsymbol{x}} \Bigg(
	\frac{1}{\pi z}
	\int\displaylimits_0^\infty
	\frac{e^{- \frac{3}{2} \pi z x_2^2 w_1}}{(w_1 +1)^{\frac{3}{2}}}
	\int\displaylimits_{w_1}^\infty
	\frac{e^{- \frac{1}{2} \pi z  (2 x_1 - x_2)^2 w_2 } }{(w_2 +1 )^{\frac{3}{2}}}
	d w_2
	d w_1 
	\\
	&\qquad\qquad\qquad\qquad
	- \frac{2}{\pi z}
	\int\displaylimits_0^\infty
	\frac{e^{-2\pi z Q^*(\boldsymbol{x}) w}}{(w+1)^2}
	d w
	- 2 (2x_1 - x_2)^2
	\int\displaylimits_0^\infty
	\frac{e^{-2\pi z Q^*(\boldsymbol{x}) w}}{w+1}
	d w
	+ \frac{2}{\pi z}
	\Bigg).
	\notag
	\end{align*}
	Employing
	$
	\sideset{}{^a}\sum_{\!\!\!\!\boldsymbol{x}} (2x_1 - x_2)^2 =
	2  Q^*(\boldsymbol{x})
	$
	and integrating the third term by parts it is not hard to see that the contribution of the second line vanishes.
	Finally, making the change of variable $w_j \mapsto \frac{w_j}{z}$,
	gives the claim. 	
\end{proof}

It is also convenient to define a regularized version of the function $F$
\begin{equation*}
F^{\rm reg}(\boldsymbol{x}) := \int^{\rm reg}_{\IR^2} 
\frac{(w_1 + w_2)^2 e^{-\frac{2}{3} \pi z  Q(\boldsymbol{w}) } }
{\lp w_1 - i (2x_2-x_1) \rp \lp w_2 - i(2x_1-x_2) \rp} \boldsymbol{dw},
\end{equation*}
where for a function $f: \mathbb R^2\to \mathbb R$, we set
\begin{equation*}
\int^{\rm reg}_{\IR^2}  f(\boldsymbol{w})  \boldsymbol{dw}
\coloneqq
\frac{1}{2} \int_{\IR^2}  
\lp f(w_1, w_2) + f(-w_1, w_2) \rp \boldsymbol{dw}.
\end{equation*}
Clearly, for $2x_1-x_2,2x_2-x_1\neq 0$, we have that $F^{\text{reg}}(\boldsymbol{x})=F(\boldsymbol{x})$. Moreover, the function $F^{\rm reg}$ has removable singularities at $2x_1-x_2=0$ and $2x_2-x_1= 0$ so it extends $F$ to these values. Lemma \ref{eq:g_r1zero} still holds true at $2x_1-x_2=0$, $2x_2-x_1= 0$, or $x_1 + x_2 = 0$ by continuity with $F$ replaced by $F^{\text{reg}}$. This can be proved using Lebesgue's dominated convergence theorem.

We are now ready to prove Theorem \ref{thm2dim}.

\begin{proof}[Proof of Theorem \ref{thm2dim}]
	Starting with the definition of $\mathcal{E}_{2,\nu,-\frac{h'}{k}}$  and changing variables $w_j \mapsto i w_j - \frac{h'}{k}$, we rewrite $\mathcal{E}_{2,\nu,-\frac{h'}{k}}( \frac{h'}{k}  + i z)$ as
	\begin{equation}\label{Thetaint}
	-
	\int\displaylimits_0^\infty \frac{1}{(w_1 + z)^{\frac{3}{2}}}
	\int\displaylimits_{w_1}^\infty \frac{1}{(w_2 + z)^{\frac{3}{2}}}
	\sum_{\a    \pmod{2} }
	\vartheta_{\frac{2 \nu + 3 \a}{6}} \lp 3 \lp i w_1 - \frac{h'}{k} \rp \rp  
	\vartheta_{\frac{\a}{2}} \lp i w_2 - \frac{h'}{k}\rp
	d w_2  d w_1.
	\end{equation}
	Using
	\begin{equation*}
	\vartheta_{\frac{\a}{2}} \lp \t\rp
	= \sum_{n \in \a+ 2\IZ} q^{\frac{n^2}{4}}
	\andd
	\vartheta_{\frac{2 \nu + 3 \a}{6}} \lp 3 \t \rp
	= \sum_{n \in \a +\frac{2\nu}{3}+ 2\IZ  } q^{\frac{3n^2}{4}} 
	\end{equation*}
	we obtain that
	\begin{multline}\label{sumTheta}
	\sum_{\alpha   \pmod{2}}
	\vartheta_{\frac{2 \nu + 3 \a}{6}} \left(3 \left( i w_1 - \frac{h'}{k} \right)\right)  
	\vartheta_{\frac{\a}{2}} \left(i w_2 - \frac{h'}{k}\right)\\
	=
	\sum_{\a    \pmod{2} }
	\sum_{\substack{  n_1 \in  \a + \frac{2\nu}{3}+ 2\IZ \\ n_2 \in \a + 2\IZ   } }
	\zeta_{4k}^{-h'\left(3n_1^2+n_2^2\right)} 
	e^{-\frac{3}{2} \pi n_1^2 w_1-\frac{1}{2} \pi n_2^2 w_2}.
	\end{multline}
	Changing variables to $m_1 = \frac{n_1+n_2}{2}$ and $m_2 = n_1$ so that $m_1$ runs over $\frac{\nu}{3}+\IZ$ and $m_2$ runs over $\a + \frac{2\nu}{3}+2 \IZ $, we rewrite \eqref{sumTheta} as
	\begin{equation*}\label{summ12}
	\sum_{\substack{  m_1 \in  \frac{\nu}{3} +\IZ   \\ m_2 \in  -\frac{\nu}{3}+\IZ     } }
	\zeta_k^{-h'Q^*(\boldsymbol{m})}  
	e^{-\frac{3\pi}{2} m_2^2 w_1-\frac{\pi}{2} (2m_1-m_2)^2 w_2}=\sum_{\substack{  m_1 \in  \frac{\nu}{3} +\IZ   \\ m_2 \in  -\frac{\nu}{3}+\IZ     } }
	\zeta_k^{-h'Q^*(\boldsymbol{m})}  \sideset{}{^a}\sum_{\boldsymbol{m}}
	e^{-\frac{3\pi}{2}  m_2^2 w_1-\frac{\pi}{2} (2m_1-m_2)^2 w_2},
	\end{equation*}
	using that the set over which $m_1$ and $m_2$ are summed as well as the root of unity inside are both invariant under the involutions $\iota_1$ and $\iota_2$.
 We now interchange in \eqref{Thetaint} the outer sum with the integrals and find that
\begin{equation}\label{eq:E2withG}
\mathcal{E}_{2,\nu,-\frac{h'}{k}} \lp \frac{h'}{k}  + i z \rp
=- \frac{2}{\sqrt{3}}  
\sum_{\substack{  m_1 \in \IZ  + \frac{\nu}{3} \\ m_2 \in \IZ - \frac{\nu}{3}    } }
\zeta_k^{- h' Q^*(\boldsymbol{m})}
\sideset{}{^a}\sum_{\boldsymbol{m}}
G(\boldsymbol{m}; z).
\end{equation}
This interchange is legal due to Fubini's Theorem because the double series on the left-hand side of equation \eqref{eq:E2withG} is absolutely convergent when the integrand in the definition of $G(\boldsymbol{m}; z)$ is replaced with its absolute value.
Using Lemma \ref{lem:integral_identity} then gives
\begin{equation*}
\mathcal{E}_{2,\nu,-\frac{h'}{k}} \lp \frac{h'}{k}  + i z \rp
=- \frac{2}{\sqrt{3}}  
\sum_{\substack{  m_1 \in \IZ  + \frac{\nu}{3} \\ m_2 \in \IZ - \frac{\nu}{3}    } }
\zeta_k^{- h' Q^*(\boldsymbol{m})}
\sideset{}{^a}\sum_{\boldsymbol{m}}
F^{\rm reg}(\boldsymbol{m}; z).
\end{equation*}
Again using that the double series outside is absolutely convergent we can	
 change variables to $n_1 = 2 m_2 - m_1$ and $n_2 = 2m_1 - m_2$ so that $n_1, n_2$ run through integers satisfying $n_1 \equiv n_2 + \nu \pmod{3}$.
	Note that $\iota_1$ corresponds to $\kappa_1:(n_1,n_2) \mapsto (-n_2, -n_1)$ and $\iota_2$ corresponds to
	$\kappa_2:(n_1,n_2) \mapsto (n_1+n_2, -n_2)$. Thus the averaging sum over $(m_1, m_2)$ becomes averaging $(n_1, n_2)$ over the orbit
	\begin{align*}
	(n_1, n_2) 
	&\xmapsto{\kappa_1} (-n_2, -n_1)
	\xmapsto{\kappa_2} (-n_1-n_2, n_1)
	\xmapsto{\kappa_1} (-n_1, n_1+n_2)
	\xmapsto{\kappa_2} (n_2, -n_1-n_2)\\
	&\xmapsto{\kappa_1} (n_1+n_2, -n_2)
	\xmapsto{\kappa_2} (n_1, n_2).
	\end{align*}
	Noting that the integrals corresponding to $(n_1, n_2)$ and to $(-n_2, -n_1)$ are the same, we obtain\footnote{
Each of the three terms in the parenthesis naively gives equal contributions as they seem to be related by a change of dummy variables for the double sum. This however leads to a wrong result; because if the terms are separated, one only gets conditionally convergent double series. The ordering in the double sum that gives convergence should be picked to be the same for each of the three terms and that ordering is not necessarily compatible with the change of variables required to show that the contribution from each term is equal.	
	}
	\begin{align*}
	&\mathcal{E}_{2,\nu,-\frac{h'}{k}} \lp \frac{h'}{k}  + i z \rp
	=\frac{2}{3\sqrt{3}}  
	\sum_{\substack{  n_1, n_2 \in \IZ   
			\\ n_1 \equiv n_2 + \nu    \pmod{3}   }}
	\zeta_{3k}^{- h' Q(\boldsymbol{n})}
	\int^{\rm reg}_{\IR^2} 
	(w_1 + w_2)^2 e^{-\frac{2}{3} \pi z Q(\boldsymbol{w}) }
	\\
	&
	\quad  
	\times 
	\lp 
	\frac{1}{\lp i w_1 + n_1 \rp \lp i w_2 + n_2 \rp}
	+ \frac{1}{\lp i w_1 - n_1 \rp \lp i w_2 + n_1 + n_2 \rp}
	+ \frac{1}{\lp i w_1 + n_1 + n_2 \rp \lp i w_2 - n_2 \rp}
	\rp
	\boldsymbol{dw}.
	\end{align*}
We then write $n_j = r_j + 3 k m_j$ where $r_j$ runs modulo $3k$ and pick a particular order for the sums over $m_j$ that makes the individual terms convergent to obtain that
	\begin{align*}
	\mathcal{E}_{2,\nu,-\frac{h'}{k}} \left(\frac{h'}{k}  + i z\right)=\frac{2}{3\sqrt{3}} 
	\sum_{\substack{  r_1, r_2 \!\!\!\pmod{3k} 
			\\ r_1 \equiv r_2 + \nu \!\!\! \pmod{3}   } }
	\zeta_{3k}^{- h' Q(\boldsymbol{r})}
	\int^{\rm reg}_{\IR^2} 
	e^{-\frac{2}{3} \pi z Q(\boldsymbol{w})} (w_1 + w_2)^2\sum_{j=1}^3 \mathcal C_{\boldsymbol{r},j}(\boldsymbol{w})\boldsymbol{dw},
	\end{align*}
	where
	\begin{align*}
	\mathcal C_{\boldsymbol{r},1}(\boldsymbol{w}) &:= \lim_{M_1,M_2\to\infty} \sum_{m_1=-M_1}^M \sum_{m_2=-M_2}^{M_2} \frac{1}{(iw_1+r_1+3km_1)(iw_2+r_2+3km_2)},\\
	\mathcal C_{\boldsymbol{r},2}(\boldsymbol{w}) &:= \lim_{M_1,M_2\to\infty} \sum_{m_1=-M_1}^M \sum_{m_2=-M_2}^{M_2}\frac{1}{(iw_1-r_1-3km_1)(iw_2+r_1+r_2+3k(m_1+m_2))},\\
	\mathcal C_{\boldsymbol{r},3}(\boldsymbol{w}) &:= \lim_{M_1,M_2\to\infty} \sum_{m_1=-M_1}^M \sum_{m_2=-M_2}^{M_2}\frac{1}{(iw_1+r_1+r_2+3k(m_1+m_2))(iw_2-r_2-3km_2)}.\\
	\end{align*}
	Using \eqref{cotid}, we obtain
	\begin{equation*}
	\mathcal C_{\boldsymbol{r},1}(\boldsymbol{w}) = \frac{\pi^2}{9k^2} \cot\left(\pi\frac{iw_1+r_1}{3k}\right)\cot\left(\pi\frac{iw_2+r_2}{3k}\right).
	\end{equation*}
	To compute $\mathcal C_{\boldsymbol{r},2}$, we observe that
	\begin{equation*}
	\mathcal C_{\boldsymbol{r},2}(\boldsymbol{w}) = \mathcal C_{\kappa_1\circ\kappa_2\circ\kappa_1(\boldsymbol{r}),1}(\boldsymbol{w}).
	\end{equation*}
	Thus these terms yield the same contribution to the overall sum. 
	For $\mathcal C_{\boldsymbol{r},3}$, we split
	\begin{multline*}
	\frac{1}{(iw_1+r_1+r_2+3k(m_1+m_2))(iw_2-r_2-3km_2)} \\
	= \left(\frac{1}{iw_1+r_1+r_2+3k(m_1+m_2)}+\frac{1}{iw_2-r_2-3km_2}\right)\frac{1}{i(w_1+w_2)+r_1+3km_1},
	\end{multline*}
	to obtain that
	\begin{equation*}
	\mathcal C_{\boldsymbol{r},3}(w_1,w_2) = -\mathcal C_{\kappa_2\circ\kappa_1(\boldsymbol{r}),1}(-w_1,w_1+w_2)-\mathcal C_{\boldsymbol{r},1}(w_1+w_2,-w_2).
	\end{equation*}
	For the contribution from the first term, we change variables $(w_1,w_2)\mapsto (-w_1,w_1+w_2)$ and for the contribution from the second term, we change variables $(w_1,w_2)\mapsto (w_1+w_2,-w_2)$ yielding that $\mathcal E_{2,\nu,-\frac{h'}{k}}(\frac{h'}{k}+iz)$ equals
	\begin{multline*}
	\frac{2 \pi^2}{27\sqrt{3}k^2} 
	\sum_{\substack{  r_1, r_2    \pmod{3k}
			\\ r_1 \equiv r_2 + \nu    \pmod{3}   } }
	\zeta_{3k}^{- h' Q(\boldsymbol{r})}\\
	\times\int^{\rm reg}_{\IR^2} 
	\left(w_1^2+w_2^2+4w_1w_2\right)
	\cot \lp  \pi \frac{i w_1 + r_1}{3k} \rp  \cot \lp  \pi \frac{i w_2 + r_2}{3k} \rp e^{-\frac{2}{3} \pi z  Q(\boldsymbol{w}) } 
	\boldsymbol{dw}.
	\end{multline*}
	Next, using \eqref{cotid2},
	we rewrite 
	\begin{multline*}
	\cot \left(\pi \frac{i w_1 + r_1}{3k} \right)\cot \left(\pi \frac{i w_2 + r_2}{3k}\right)=\frac{
		\sin \lp \frac{2 \pi r_1}{3k} \rp   \sin \lp \frac{2 \pi r_2}{3k} \rp
		- \sinh \lp \frac{2 \pi w_1}{3k} \rp   \sinh \lp \frac{2 \pi w_2}{3k} \rp
	}{\lp \cosh \lp \frac{2 \pi w_1}{3k} \rp - \cos \lp \frac{2 \pi r_1}{3k} \rp \rp  
		\lp \cosh \lp \frac{2 \pi w_2}{3k} \rp - \cos \lp \frac{2 \pi r_2}{3k} \rp \rp}\\
	- i\frac{\sinh \lp \frac{2 \pi w_1}{3k} \rp  \sin \lp \frac{2 \pi r_2}{3k} \rp
		+\sinh \lp \frac{2 \pi w_2}{3k} \rp  \sin \lp \frac{2 \pi r_1}{3k} \rp}{\lp \cosh \lp \frac{2 \pi w_1}{3k} \rp - \cos \lp \frac{2 \pi r_1}{3k} \rp \rp  
		\lp \cosh \lp \frac{2 \pi w_2}{3k} \rp - \cos \lp \frac{2 \pi r_2}{3k} \rp \rp}.
	\end{multline*}
	The contribution of the imaginary part to the integral vanishes (it is odd under the change of variables $(w_1, w_2) \mapsto (-w_1, -w_2)$ and the rest of the integrand is even) and we find that
	\begin{equation*}
	\mathcal{E}_{2,\nu,-\frac{h'}{k}} \lp \frac{h'}{k} + iz \rp = 
	-\frac{2 \pi^2}{27 \sqrt{3}k^2} 
	\sum_{\substack{ r_1, r_2    \pmod{3k} \\ r_1 \equiv r_2 + \nu    \pmod{3} }}
	\zeta_{3k}^{-h'Q(\boldsymbol{r})} 
	\mathcal F_k (\boldsymbol{r}; z), 
	\end{equation*}
	where we define 
	\begin{multline*}
	\mathcal F_k (\boldsymbol{r}; z)\coloneqq 
	\frac{1}{2} \int\displaylimits_{\IR^2} 
	\sum_{\pm}\left(-f_{\frac{r_1}{3k}}\left(\frac{w_1}{k}\right)f_{\frac{r_2}{3k}}\left(\frac{w_2}{k}\right)\pm g_{\frac{r_1}{3k}}\left(\frac{w_1}{k}\right)g_{\frac{r_2}{3k}}\left(\frac{w_2}{k}\right)\right)\\
	\times  \left(w_1^2+w_2^2\pm4w_1w_2\right) e^{-\frac{2}{3} \pi z  \left(w_1^2+w_2^2\pm w_1w_2\right) }
	\boldsymbol{dw}.
	\end{multline*}
	If $r_1 , r_2 \not\equiv 0 \pmod{3k}$, then the two terms that contribute to the sum can be separately integrated and are equal to each other, so we obtain the claim, changing $w_1$ into $-w_1$ for the minus sign.
	
	Next suppose that $r_1 \equiv 0$, $r_2 \not\equiv 0 \pmod{3k}$. In what follows, we add and subtract terms which allows us to separate several terms that are well-behaved near $w_1 = 0$ and can be integrated individually. 
	By definition 
	\[
	\mathcal F_k(0,r_2;z)=\frac12\int_{\mathbb R^2}e^{-\frac{2}{3} \pi z 
	\left(w_1^2+w_2^2\right)}g_0\left(\frac{w_1}{k}\right)g_{\frac{r_2}{3k}}\left(\frac{w_2}{k}\right)\sum_{\pm}\left(4w_1w_2\pm w_1^2\pm w_2^2\right)e^{\mp\frac{2}{3} \pi z w_1w_2}\boldsymbol{dw}.
	\]
	The term $4w_1w_2\pm w_1^2$ contributes, again changing $w_1$ into $-w_1$ for the minus sign,
	\begin{equation}\label{cont1}
	\int_{\mathbb R^2}w_1(4w_2+w_1) g_0\left(\frac{w_1}{k}\right)g_{\frac{r_2}{3k}}\left(\frac{w_2}{k}\right)e^{-\frac{2}{3} \pi z Q(\boldsymbol{w})}\boldsymbol{dw}.
	\end{equation}
	In the term from $\pm w_2^2$, we write
	\[
	g_0\left(\frac{w_1}{k}\right)=\left(g_0\left(\frac{w_1}{k}\right)-\frac{3k}{\pi w_1}\right)+\frac{3k}{\pi w_1}.
	\]
	The first term has a removable singularity and contributes, changing $w_1$ into $-w_1$, for the minus sign
	\begin{equation}\label{cont2}
	\int_{\mathbb R^2}w_2^2\left(g_0\left(\frac{w_1}{k}\right)-\frac{3k}{\pi w_1}\right)g_{\frac{r_2}{3k}}\left(\frac{w_2}{k}\right)e^{-\frac{2}{3} \pi z Q(\boldsymbol{w})}\boldsymbol{dw}.
	\end{equation}
	In the term from $\frac{3k}{\pi w_1}$, we write
	\begin{equation*}
	G_{\frac{r_2}{3k}}\left(\frac{w_2}{k}\right)=\left(G_{\frac{r_2}{3k}}\left(\frac{w_2}{k}\right)-G_{\frac{r_2}{3k}}\left(\frac{w_2\pm\frac{w_1}{2}}{k}\right)\right)+G_{\frac{r_2}{3k}}\left(\frac{w_2\pm\frac{w_1}{2}}{k}\right),
	\end{equation*}
	where
	\[
	G_c(x) \coloneqq x^2g_c(x).
	\]
	The first term contributes, again changing $w_1$ into $-w_1$ for the minus term
	\begin{equation}\label{cont3}
	\frac{3k^3}{\pi}\int_{\mathbb R^2}w_1^{-1}\left(G_{\frac{r_2}{3k}}\left(\frac{w_2}{k}\right)-G_{\frac{r_2}{3k}}\left(\frac{w_2+\frac{w_1}{2}}{k}\right)\right)e^{-\frac{2}{3} \pi z Q(\boldsymbol{w})}\boldsymbol{dw}.
	\end{equation}
	Using $w_1^2+w_2^2\pm w_1w_2 =( w_2\pm \frac{w_1}{2} )^2+\frac34w_1^2$, the contribution from the final term can be written as
	\[
	\frac{3k^3}{2\pi}\int_{\mathbb R}w_1^{-1}e^{-\frac{3}{4} \pi z w_1^2}\left(\int_{\mathbb R}\sum_{\pm}\pm G_{\frac{r_2}{3k}}\left(\frac{w_2\pm\frac{w_1}{2}}{k}\right)e^{-\frac{2}{3} \pi z \left(w_2\pm\frac{w_1}{2}\right)^2}dw_2\right)dw_1.
	\]
	The integral on $w_2$ vanishes as may be seen by changing variable $w_2 \mapsto w_2 + w_1$ in the integral for the minus sign. Combining gives the claim.
	The case $r_1\not\equiv 0,\ r_2\equiv 0\pmod{3k}$ is completely analogous.
	
	Finally, we assume $r_1\equiv r_2\equiv 0\pmod{3k}$. The term coming from $4w_1w_2$ contributes, again changing $w_1$ into $-w_1$ in the term with the minus sign 
	\[
	4\int_{\mathbb R^2}w_1w_2g_0\left(\frac{w_1}{k}\right)g_0\left(\frac{w_2}{k}\right)e^{-\frac{2}{3} \pi z Q(\boldsymbol{w})}\boldsymbol{dw}.
	\]
		In the contribution from $w_2^2$, the function $w_2\mapsto w_2^2g_0\left(\frac{w_2}{k}\right)$ has a removable singularity and the exact same proof as in the case $r_1\equiv 0,\ r_2\not\equiv 0\pmod{3k}$ works for handling $g_0\left(\frac{w_1}{k}\right)$. 
In the contribution from $w_1^2$, we switch roles of $w_1$ and $w_2$. This gives overall
	\begin{equation*}
	\int_{\mathbb R^2}\mathcal G_k(\boldsymbol{w}) e^{-\frac{2}{3} \pi z Q(\boldsymbol{w})}\boldsymbol{dw},
	\end{equation*}
	where
	\begin{align}\label{split0}
	\mathcal G_k(\boldsymbol{w})& \coloneqq 4w_1w_2g_0\left(\frac{w_1}{k}\right)g_0\left(\frac{w_2}{k}\right)+k^2G_0\left(\frac{w_2}{k}\right)\left(g_0\left(\frac{w_1}{k}\right)-\frac{3k}{\pi w_1}\right)\\
	&\qquad+\frac{3k^3}{\pi w_1}\left(G_0\left(\frac{w_2}{k}\right)-G_0\left(\frac{w_2+\frac{w_1}{2}}{k}\right)\right)+k^2G_0\left(\frac{w_1}{k}\right)\left(g_0\left(\frac{w_2}{k}\right)-\frac{3k}{\pi w_2}\right)
\notag	
	\\
	&\qquad+\frac{3k^3}{\pi w_2}\left(G_0\left(\frac{w_1}{k}\right)-G_0\left(\frac{w_1+\frac{w_2}{2}}{k}\right)\right).
	\notag
	\end{align}
	Simplifying gives the claim.
\end{proof}
\section{Bounds for Eichler integrals}\label{sec5}
In this chapter we find bounds for the Eichler integrals \eqref{defineTint} and \eqref{defineTint2} using the representations from Lemma \ref{lem3.1} and Theorem \ref{thm2dim}. For this, we split off ``principal part contributions''.

\subsection{The one-dimensional case}
Define, for $b\geq 0$, 
\[
\mathcal E^*_{1,j,-\frac{h'}{k},b}\left(\frac{h'}{k}+ \frac{i}{z}  \right)
\coloneqq
e^{\frac{2\pi b}{z}}
\frac{\pi i}{3\sqrt{6}k}
\sum_{\substack{r\pmod{6k}\\r\equiv j  \pmod{6}}}
\zeta^{-h'r^2}_{12k}\int_{-2\sqrt{3b}}^{2\sqrt{3b}}wg_{\frac{r}{6k}}\left(\frac{w}{2k}\right)e^{-\frac{\pi w^2}{6 z}}dw.
\]
\begin{lem}\label{lem-onedim}
	We have, for $b\geq 0$, $h',k\in\IZ$, $k>0$,  and  $\textnormal{Re}(\frac{1}{z})  \geq  1$,
	\[
	e^{\frac{2\pi b}{z}}
	\mathcal E_{1,j,-\frac{h'}{k}}\left(\frac{h'}{k}+\frac{i}{z}\right)
	=
	\mathcal E^*_{1,j,-\frac{h'}{k},b}\left(\frac{h'}{k}+\frac{i}{z}\right)+O(\log(k)),
	\]
	where the error term is independent of $h'$ and $z$.
\end{lem}
\begin{proof}
	We compute
	\begin{multline*}
	\left|e^{\frac{2\pi b}{z}}
	\mathcal E_{1,j,-\frac{h'}{k}}\left(\frac{h'}{k}+\frac{i}{z}\right)
	-\mathcal E^*_{1,j,-\frac{h'}{k},b}\left(\frac{h'}{k}+\frac{i}{z}\right)\right|
	\\\ll
	e^{2\pi b\mathrm{Re}\left(\frac1z\right)}\frac{1}{k}\sum_{r=0}^{6k-1}
	\int\displaylimits_{|w|\geq 2\sqrt{3b}} 
	\left| w g_{\frac{r}{6k}}\left(\frac{w}{2k}\right)  \right|
	e^{-\frac{1}{6} \pi \text{Re}\left(\frac1z\right)w^2}dw.
	\end{multline*}
	We first bound the contribution from $r\neq 0$.
	 Using that, for $0<c<1$, 
	 \begin{equation}\label{gbound}
	| g_c(w) |
	\ll \frac{1}{c}   + \frac{1}{1-c},
	 \end{equation}
	 it is not hard to see that this contribution is $O(\log(k))$.
	 
	For $r= 0$, we 
	require that 
	\begin{equation}\label{gobound}
	| w g_0(w) |
	\ll 1+ |w |,
	\end{equation}
	to show that this term contributes $O(1)$. Combining gives the statement of the lemma.
\end{proof}

\subsection{The two-dimensional case}

Define, for $b \geq 0$,
\begin{align*}
&\mathcal{E}^*_{2, \nu, - \frac{h'}{k}, b} \lp \frac{h'}{k}+ \frac{i}{z} \rp
\notag:=
- e^{\frac{2 \pi b}{z}}   \frac{2 \pi^2}{27 \sqrt{3} k^2}  
\sum_{\substack{r_1, r_2    \pmod{3k} \\ r_1 \equiv r_2 + \nu    \pmod{3} }}
\zeta_{3k}^{- h' Q(\boldsymbol{r})}\ \ 
\int\displaylimits_{\mathclap{Q(\boldsymbol{w}) \leq 3 b}} 
g_{k,\boldsymbol{r}} (\boldsymbol{w}) 
e^{-\frac{2\pi}{3z} Q(\boldsymbol{w})}   
\boldsymbol{dw}.
\end{align*}

\begin{lem}\label{lem-twodim}
	For $b \geq 0$, $h',k\in\IZ$, $k>0$, and $\mathrm{Re}(\frac{1}{z})\geq 1$ we have
	\begin{equation*}
	e^{\frac{2\pi b}{z}}  \,  
	\mathcal{E}_{2, \nu, - \frac{h'}{k}} \lp \frac{h'}{k}+ \frac{i}{z} \rp
	= 
	\mathcal{E}^*_{2, \nu, - \frac{h'}{k}, b} \lp \frac{h'}{k}+ \frac{i}{z} \rp
	+ \mathrm{O} \lp\log (k)^2 \rp,
	\end{equation*}
	where the error term is independent of $h'$ and $z$.
\end{lem}
\begin{proof}
	We first bound
	\begin{multline*}
	\left|
	e^{\frac{2\pi b}{z}}   
	\mathcal{E}_{2, \nu, - \frac{h'}{k}} \lp \frac{h'}{k}+ \frac{i}{z} \rp
	-
	\mathcal{E}^*_{2, \nu, - \frac{h'}{k}, b} \lp \frac{h'}{k}+ \frac{i}{z} \rp
	\right|
	\notag \\
	\ll
	\frac{e^{2 \pi b \mathrm{Re} \lp \frac{1}{z} \rp} }{k^2}
	\sum_{\boldsymbol{r} \pmod{3k}} \ 
	\int\displaylimits_{Q(\boldsymbol{w}) \geq 3 b}
	|  g_{k,\boldsymbol{r}} (\boldsymbol{w}) | 
	e^{-\frac{2}{3} \pi \mathrm{Re} \lp \frac{1}{z} \rp Q(\boldsymbol{w})}   
	\boldsymbol{dw}.
	\end{multline*}
Using polar coordinates, it is not hard to show that, for $j_1, j_2 \in \IN_0$, we have
	\begin{equation*}\label{claim1}
	\int\displaylimits_{\mathclap{Q(\boldsymbol{w}) \geq 3 b}} 
	| w_1|^{j_1}   | w_2|^{j_2} 
	e^{-\frac{2}{3} \pi \mathrm{Re} \lp \frac{1}{z} \rp Q(\boldsymbol{w})}   
	\boldsymbol{dw}
	\ll  e^{-2 \pi b \mathrm{Re} \lp \frac{1}{z} \rp} .
	\end{equation*}

We now first bound the contribution from $r_1,r_2\not\equiv 0\pmod{3k}$. By \eqref{gbound} and the fact that $f_c$ is maximized at $w=0$, one may show that this contribution is $O(\log(k)^2)$.

We next consider the case $r_1\equiv 0, r_2\not\equiv 0\pmod{3k}$.  For this, we split $g_{k,(0,r_2)}$ as in \eqref{cont1}, \eqref{cont2}, and \eqref{cont3}. We first bound \eqref{cont1}, using \eqref{gbound} and \eqref{gobound},
\[
\left| w_1(4w_2+w_1)g_0\left(\frac{w_1}{k}\right)g_{\frac{r_2}{3k}}\left(\frac{w_2}{k}\right) \right|
\ll k\left(1+\frac{|w_1|}{k}\right) (|w_1|+|w_2| )
\left(\frac{k}{r_2} +  \frac{k}{3k - r_2}     \right).
\]
For \eqref{cont2}, we use the bound
\begin{equation}\label{primbound}
g_0(w)-\frac{3}{\pi w}\ll 1,
\end{equation}
and estimate 
\[
\left| 
w_2^2 \, 
g_{\frac{r_2}{3k}}
\left(\frac{w_2}{k}\right)\left(g_0\left(\frac{w_1}{k}\right)-\frac{3k}{\pi w_1}\right)
\right|
\ll
\left(\frac{k}{r_2} +  \frac{k}{3k - r_2}   \right)
w_2^2.
\]
For \eqref{cont3}, we use Taylor expansions, to bound
\begin{equation*}
\frac{1}{|w_1|}\left|G_c(w_2)-G_c\left(w_2+\frac{w_1}{2}\right)\right| \ll\left(w_1^2+w_2^2+|w_2|\right)\left(\frac1{c^2}+\frac{1}{(1-c)^2}\right).
\end{equation*}
Thus we may estimate
\begin{align*}
\frac{k^3}{|w_1|}\left|G_{\frac{r_2}{3k}}\left(\frac{w_2}{k}\right)-G_{\frac{r_2}{3k}}\left(\frac{w_2+\frac{w_1}{2}}{k}\right)\right|\ll \left( w_1^2 + w_2^2 + |w_2| \right) \left(\frac{k^2}{r_2^2} +\frac{k^2}{(3k-r_2)^2}  \right)
.
\end{align*}
Therefore 
\begin{align*}
| g_{k,(0,r_2)}(\boldsymbol{w})|
&\ll 
k\left( 1+\frac{|w_1|}{k}\right)  (|w_1|+|w_2|)
\left(\frac{k}{r_2} +  \frac{k}{3k - r_2}     \right)\\
&\quad
+
\left( w_1^2 + w_2^2 + |w_2| \right) \left(\frac{k^2}{r_2^2} +\frac{k^2}{(3k-r_2)^2}  \right)
\end{align*}
and we obtain the overall contribution as $O(\log(k))$.
The case $r_1\not\equiv  0\pmod{3k}$, $r_2 \equiv  0\pmod{3k}$ is done in exactly the same way.

We finally consider the case $r_1\equiv r_2\equiv 0\pmod{3k}$. Using the splitting as in \eqref{split0} and employing \eqref{gobound}, we may bound the first term of $\mathcal G_k$ against
\begin{equation*} k^2\left(1+\frac{|w_1|}{k}\right)\left(1+\frac{|w_2|}{k}\right).
\end{equation*}
Overall this term contributes $O(1)$.

The second term of $\mathcal G_k$ is estimated against, using \eqref{gobound} and \eqref{primbound}
\begin{equation*}
k^2 \frac{|w_2|}{k}\left(1+\frac{|w_2|}{k}\right).
\end{equation*}
Overall this term contributes $O(\frac{1}{k})$.
The fourth term is handled in exactly in the same way. 

The third term of $\mathcal G_k$ may be bounded by
\begin{equation*}\label{third0} 
\frac{k^3}{|w_1|}\left|G_0\left(\frac{w_2}{k}\right)-G_0\left(\frac{w_2+\frac{w_1}{2}}{k}\right)\right|\ll k^2(1+|w_1|),
\end{equation*}
again using Taylor's Theorem.
Thus overall this term contributes $O(1)$.
The  fifth term is handled in exactly the same way. 

Combining gives the claim.
\end{proof}

\section{The Circle Method and the proof of Theorem \ref{main theorem} and Corollary \ref{ascor}}\label{sec6}
\subsection{Proof of Theorem \ref{main theorem}}
We follow the version of the Circle Method due to Rademacher and refer the reader to Chapter 5 of \cite{AP} for basic facts on Farey fractions and the Circle Method. 
The starting point is Cauchy's Theorem, which yields
\[
a_{3,\mu}(n)=\int_i^{i+1}f_{3,\mu}(\tau)e^{-2\pi in_\mu\tau}d\tau,
\]
where the integral goes along any path connecting $i$ and $i+1$. We decompose the integral into arcs lying near the root of unity $\zeta_k^h$, where $0\leq h<k\leq N$ with $\mathrm{gcd}(h,k)=1$, and $N\in\mathbb N$ is a parameter, which then tends to infinity. For this, the {\it Ford Circle} $\mathcal C_{h,k}$ denotes the circle in the complex $\tau$-plane with radius
$\frac{1}{2k^2}$ and center $\frac{h}{k}+\frac{i}{2k^2}$.
We let $P_N := \bigcup_{\frac{h}{k} \in F_N} C_{h,k}(N)$, where $F_N$ is the Farey sequence of order $N$ and $C_{h,k}(N)$ is the upper arc of the Ford Circle $\mathcal C_{h,k}$ from its intersection with $\mathcal C_{h_1,k_1}$ to its intersection with $\mathcal C_{h_2,k_2}$ where $\frac{h_1}{k_1} < \frac{h}{k} < \frac{h_2}{k_2}$ are consecutive fractions in $F_N$. In particular $C_{0,1}(N)$ and $C_{1,1}(N)$ are half-arcs with the former starting at $i$ and the latter ending at $i+1$. This is illustrated for $P_4$ in Figure \ref{fig_Rademacher_path}.
\begin{minipage}{\linewidth}
\begin{center}
\centering
\includegraphics[width=0.4\linewidth]{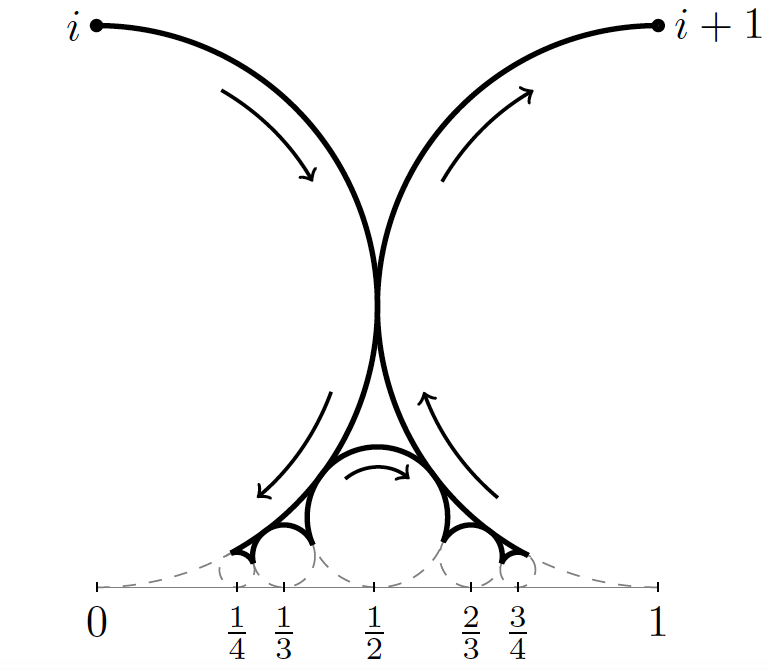}
\captionof{figure}{Rademacher's integration path $P_N$ for $N=4$.}
\label{fig_Rademacher_path}
\end{center}
\end{minipage} We obtain
\begin{equation}\label{formulacoeff}
\a_{3,\mu} (n) = 
\sum_{k=1}^N \sum_{\substack{0 \leq h \leq k \\ \mathrm{gcd}(h,k)=1 }}\ 
\int\displaylimits_{C_{h,k}(N)} 
 f_{3,\mu}(\t)
 e^{-2 \pi i n_\mu\t}
d \t.
\end{equation}

Next, we make the change of variables $\tau=\frac{h}{k}+\frac{iz}{k^2}$, which maps the Ford Circles to a standard circle with radius $\frac12$ which is centered at $z=\frac{1}{2}$. 
The image of the arc $C_{h,k}(N)$ is now an arc on the standard circle from $z_1$ to $z_2$, where
\begin{equation*}
z_1 = z_1 (h,k ) := \frac{k^2}{k^2 + k_1^2} + i\frac{ k k_1}{k^2 + k_1^2}, \qquad
z_2 = z_2 (h,k ) := \frac{k^2}{k^2 + k_2^2} - i\frac{ k k_2}{k^2 + k_2^2} .
\end{equation*}
We also combine the half-arcs $C_{0,1}(N)$ and $C_{1,1}(N)$ into an arc in the $z$-plane from $z_1 (1,1) \coloneqq \frac{1}{1-iN}$ to $z_2 (1,1) \coloneqq \frac{1}{1+iN}$ by shifting the $C_{1,1}(N)$ half-arc as $\t \mapsto \t-1$. Note that on the disc bounded by the standard circle we always have $\text{Re}(\frac1z)\geq 1$. Moreover, for any point $z$ on the chord that is connecting $z_1 (h,k )$ and $z_2 (h,k )$, we have $| z | \leq \frac{k \sqrt{2}}{N}$ and the length of this chord does not exceed $\frac{2 \sqrt{2}k}{N}$.

Equation \eqref{formulacoeff} then becomes
\begin{equation}\label{FourierCoeff}
\a_{3,\mu} (n) = 
\sum_{k=1}^N  \frac{i}{k^2}
\sum_{\substack{0 \leq h < k \\ \mathrm{gcd}(h,k)=1 }}
\ \ \int\displaylimits_{z_1} ^{z_2}
e^{- 2 \pi i n_\mu \lp \frac{h}{k} + \frac{iz}{k^2}   \rp}  
f_{3,\mu} \lp \frac{h}{k} + \frac{iz}{k^2}  \rp
d z.
\end{equation}
Now, as $N \to \infty$ the path of integration gets closer to the point $z=0$ for each term. Using modular transformations one can control the way the integrand behaves as this happens. In particular, we use the modular transformations $M = \left(\begin{smallmatrix}h' & -\frac{1+hh'}{k} \\ k & -h\end{smallmatrix}\right)$ with $h'$ satisfying $h h' \equiv -1\pmod{k}$. Under this modular transformation we have\footnote{The Ford Circles are mapped to the $\mathrm{Im} (\t) = 1$ line under this transformation.}
$
\frac{a \t + b}{c \t + d}  = \frac{h'}{k} + \frac{i}{z}.
$
Using Lemma \ref{lem:mocktrans}, we obtain that $f_{3,\mu} ( \frac{h}{k} + \frac{iz}{k^2})$ equals
\begin{align}\notag
&\frac{z^{\frac{3}{2}}}{k^{\frac{3}{2}}}
\sum_{\nu \pmod{3}} \psi_{h,k} (\nu, \mu) \  \Bigg( 
f_{3,\nu} \lp \frac{h'}{k} + \frac{i}{z} \rp
-  \frac{9 \sqrt{3}i}{2 \sqrt{2}\pi}\sum_{\a \pmod{2}}
f_\a \lp \frac{h'}{k} + \frac{i}{z} \rp  
\mathcal{E}_{1,2\nu+3\a, -\frac{h'}{k}} \lp \frac{h'}{k} + \frac{i}{z} \rp
\notag \\
&
\qquad \qquad \qquad \qquad \qquad\qquad
- \frac{9 \sqrt{3}}{16\pi^2} 
f \lp \frac{h'}{k} + \frac{i}{z} \rp
\mathcal{E}_{2,\nu, -\frac{h'}{k}} \lp \frac{h'}{k} + \frac{i}{z} \rp
\Bigg).
\label{modtran}
\end{align}

We now approximate all $q$-series and Eichler integrals by their principal parts and show that the introduced error is neglectible. For this, 
we use \eqref{eq:h0Fourier}, \eqref{eq:h1Fourier}, \eqref{eq:h30Fourier}, and \eqref{eq:h31Fourier} to note that
\begin{equation*}
f_{3,0}^{\rm P} (\t) \coloneqq \frac{1}{9}  q^{-\frac{3}{8}},
\quad
f_{3,1}^{\rm P} (\t) \coloneqq f_{3,-1}^{\rm P} (\t)  = 0,
\quad
f_{0}^{\rm P} (\t) \coloneqq - \frac{1}{12}  q^{-\frac{3}{8}} ,
\quad
f_{1}^{\rm P} (\t) \coloneqq 0,
\quad 
f^{\rm P} (\t) \coloneqq q^{-\frac{3}{8}} ,
\end{equation*}
where $F^p$ denotes the principal (or polar) part of a $q$-series $F$. We obtain, plugging \eqref{modtran} into \eqref{FourierCoeff} and using Lemmas \ref{lem-onedim} and \ref{lem-twodim}
\begin{equation*}
a_{3,\mu}(n) = \mathscr S_1(N)+\mathscr S_2(N)+\mathscr S_3(N)+\mathcal E(N),
\end{equation*}
where
\begin{align*}
\mathscr S_1(N) &\coloneqq \frac{i}{9} \sum_{k=1}^N k^{-\frac72} \sum_{\substack{ 0 \leq h < k \\ \gcd(h,k)=1}}   \zeta_{8k}^{-3h'-8n_\mu h} 
\, \psi_{h,k}(0,\mu)
 \int\displaylimits_{z_1}^{z_2} z^{\frac32} e^{2\pi n_\mu\frac{z}{k^2}+\frac{3\pi}{4z}} dz,\\
\mathscr S_2(N) &\coloneqq - \frac{3\sqrt{3}}{8\sqrt{2}\pi} \sum_{k=1}^N k^{-\frac72} \sum_{\substack{0 \leq h< k \\ \gcd(h,k)=1}}  \zeta_{8k}^{-3h'-8n_\mu h} \\
&\hspace{4cm}\times\sum_{\nu\pmod{3}} \psi_{h,k}(\nu,\mu) \int_{z_1}^{z_2} z^{\frac32} e^{2\pi n_\mu\frac{z}{k^2}} \mathcal E_{1,2\nu,-\frac{h'}{k},\frac38}^*\left(\frac{h'}{k}+\frac{i}{z}\right)dz,\\
\mathscr S_3(N)&\coloneqq - \frac{9\sqrt{3}i}{16\pi^2} \sum_{k=1}^N k^{-\frac72} \sum_{\substack{0\leq h < k \\ \gcd(h,k)=1}}  \zeta_{8k}^{-3h'-8n_\mu h} \\
&\hspace{4cm}\times\sum_{\nu\pmod{3}} \psi_{h,k}(\nu,\mu) \int_{z_1}^{z_2} z^{\frac32} e^{2\pi n_\mu\frac{z}{k^2}} \mathcal E_{2,\nu,-\frac{h'}{k},\frac38}^*\left(\frac{h'}{k}+\frac{i}{z}\right)dz,
\end{align*}
and where the error term $\mathcal E(N)$ satisfies
\begin{equation*}
\mathcal E(N) \ll N^{-\frac32} \log(N)^2.
\end{equation*}
In particular, $\lim_{N \to \infty} \mathcal E(N)  = 0$.

In each of the integrals, we now write
\begin{equation}\label{splitin}
\int_{z_1}^{z_2} = \int_{\mathcal C}-\int_{0}^{z_1} - \int_{z_2}^0,
\end{equation}
where $\mathcal C$ denotes the entire standard circle traversed in a clockwise direction.
Note that when integrated over an arc in the standard circle, on $[0,z_1]$ and $[z_2,0]$ the same bounds hold as for the non-principal parts (note that the length of such arcs is $\ll \frac{k}{N}$). Thus, letting $N\to \infty$, 
\begin{equation*}
a_{3,\mu}(n) = \mathscr S_1+\mathscr S_2+\mathscr S_3,
\end{equation*}
where $\mathscr S_j$ is obtained from $\mathscr S_j(N)$ by only taking the first term in \eqref{splitin} and then letting $N\to \infty$. We next rewrite the integrals over $\mathcal C$ in terms of the Bessel functions. For this, we define, for $m,n>0$, $\ell\in\IR$,
\begin{equation*}
\mathcal I_{n,m,\ell}:= \int_{\mathcal C} z^\ell e^{\frac{2\pi nz}{k^2}+\frac{2\pi m}{z}} dz.
\end{equation*}
We make the change of variables $z\mapsto \frac1z$ and use the following representation for the $I$-Bessel function
\begin{equation*}
I_\kappa(x) = \frac{\left(\frac{x}{2}\right)^\kappa}{2\pi i} \int_{c-i\infty}^{c+i\infty} t^{-\kappa-1} e^{t+\frac{x^2}{4t}} dt
\end{equation*}
with $c>0$, $\re(\kappa)>0$. This yields
\begin{equation*}
\mathcal I_{n,m,\ell}=  -2\pi i\left(k\sqrt{\frac{m}{n}}\right)^{\ell+1}I_{\ell+1}\left(\frac{4\pi \sqrt{mn}}{k}\right).
\end{equation*}
Plugging this into $\mathscr S_j$ gives the statement of the theorem.

\subsection{Proof of Corollary \ref{ascor}}
	To prove Corollary \ref{ascor}, we require the following Bessel function asymptotics $(\ell\in\frac12+\mathbb Z)$, as $x\to\infty$,
	\begin{equation}\label{Bessel}
	I_\ell(x)=\frac{e^x}{\sqrt{2\pi x}} \lp 1 - \frac{4 \ell^2 -1}{8x} 
	+ O \lp \frac{1}{x^2} \rp
	\rp.
	\end{equation}
	 Because of \eqref{Bessel}, the $k=1$ terms give the leading exponential behavior. The corresponding generalized Kloostermann sum is simply
\begin{equation*}
K_1 (\mu,\nu; n, r_1, r_2) = \psi_{0,1}(\nu,\mu) = \frac{1}{\sqrt{3}}\zeta_3^{-2\mu\nu}.
\end{equation*}
We now investigate the asymptotic behavior of the $\mathscr S_j$ seperately
starting with $\mathscr S_1$. Using \eqref{Bessel} gives that, as $n \to \infty$,
	\begin{equation}\label{S1as}
	\mathscr S_1 = \frac{1}{24\sqrt{6}n^\frac32}e^{\pi\sqrt{6n}}
	\lp  1 - \frac{3}{\pi \sqrt{6n}}  + O \lp\frac{1}{n}\rp\rp.
	\end{equation}
Next we estimate $\mathscr S_2$ using the leading term in \eqref{Bessel}
\begin{equation*}
\mathscr S_2 = - \frac{27}{256\sqrt{6} n^{\frac32}}  \lp 1 + O \lp n^{-\frac12} \rp \rp
 \sum_{\nu\pmod{3}}
  \zeta_3^{-2\mu\nu} \int_{-1}^1 \left(1-w^2\right)^{-\frac14} g_{1,\nu}^*(w)  e^{\pi\sqrt{6n\left(1-w^2\right)}}dw.
\end{equation*}
Using the saddle point method and the fact that $g_{1,\nu}^*(0) = \frac{2\sqrt{2}}{\pi} \d_{\nu,0}$ we obtain
\begin{equation*}
\mathscr S_2 = - \frac{81}{32 \pi \lp 6  n \rp^{\frac74}} e^{\pi\sqrt{6n}} \lp 1 + O \lp n^{-\frac12} \rp\rp.
\end{equation*}

Next we approximate
\begin{align*}
\mathscr S_3 =& \frac{9}{512\sqrt{6} n^{\frac32}}  \lp 1 + O \lp n^{-\frac12} \rp \rp
 \sum_{\nu\pmod{3}}  \zeta_3^{-2\mu\nu}  \\ &\qquad \times
 \sum_{\substack{r_1,r_2 \pmod{3k} \\ r_1\equiv r_2+\nu \pmod{3}}}
 \int_{Q(\boldsymbol{w})\leq 1} (1-Q(\boldsymbol{w}))^{-\frac{1}{4}}
  g^*_{1,\boldsymbol{r}}(\boldsymbol{w}) e^{\pi\sqrt{6n(1-Q(\boldsymbol{w}))}}\boldsymbol{dw}.
\end{align*}
Again using the saddle point method and the fact that $g^*_{1,\boldsymbol{r}}(\boldsymbol{0}) = \frac{27}{\pi^2} \d_{\boldsymbol{r},\boldsymbol{0}}$ we obtain
\begin{equation*}
\mathscr S_3 =  \frac{27\sqrt{3}}{256 \pi^2 n^2} e^{\pi\sqrt{6n}} \lp 1 + O \lp n^{-\frac12} \rp\rp.
\end{equation*}
Combining all three terms proves the claim.

\section{Numerical Results on an Example}\label{sec7}
In this section, we give numerical data for the Rademacher expansion of $\a_{3,\mu} (n)$ given in Theorem \ref{main theorem}. Denote the contribution of the first line of $\a_{3,\mu} (n)$ expansion by $\mathcal{A}_1 (N)$, the second line by $\mathcal{A}_2 (N)$, and the third and fourth lines by $\mathcal{A}_3 (N)$ with the sum over $k$ taken from one to $N$ in all cases. In Tables \ref{tab:a30_5} and \ref{tab:a31_5} we take the $n=5$ case as an example and display how $\mathcal{A}_1 (N)+\mathcal{A}_2 (N)+\mathcal{A}_3 (N)$ approaches to $\a_{3,0} (5) = 1512$ and $\a_{3,1} (5) = 40881$, respectively.\footnote{The computation of $\mathcal{A}_1 (N)$, $\mathcal{A}_2 (N)$, and $\mathcal{A}_3 (N)$ requires $O(N)$, $O(N^2)$, and $O(N^3)$ computations involving modified Bessel functions and their integrals, respectively. Also note that the leading term ($N=1$) can be computed in constant time and our discussion on the asymptotic expansion shows that the error is exponentially suppressed as $n$ gets larger.}

\begin{table}[h]
\centering
\begin{tabular}{c | c | c | c }
  & $N=1$ & $N=2$  &   $N=3$ \\
  \hline
$\mathcal{A}_1 (N)$ &
 $21840.0401\ldots$   &  $21843.2723\ldots$   &  $21843.0363\ldots$  \\
$\mathcal{A}_2 (N)$ &
$-32806.5410\ldots$   & $-32811.3140\ldots$   & $-32810.8548\ldots$    \\
$\mathcal{A}_3 (N)$ &
 $12478.4547\ldots$  & $12480.0457\ldots$   &  $12479.8193\ldots$  \\
$\mathcal{A}_1 (N)+\mathcal{A}_2 (N)+\mathcal{A}_3 (N)$ &
$1511.9538\ldots$   &  $1512.0039\ldots$  & $1512.0008\ldots$   \\
\hline 
\bottomrule
\end{tabular}
\caption{Numerical results for $\a_{3,0} (5) = 1512$.}
\label{tab:a30_5}
\end{table}

\begin{table}[h]
\centering
\begin{tabular}{c | c | c | c }
  & $N=1$ & $N=2$ & $N=3$  \\
  \hline
$\mathcal{A}_1 (N)$ &
$221918.638\ldots$  &  $221910.095\ldots$  & $221910.095\ldots$   \\
$\mathcal{A}_2 (N)$ &
 $-255562.432\ldots$  &  $-255548.451\ldots$  & $-255548.537\ldots$   \\
$\mathcal{A}_3 (N)$ &
$74525.064\ldots$   & $74519.364\ldots$   & $74519.440\ldots$   \\
$\mathcal{A}_1 (N)+\mathcal{A}_2 (N)+\mathcal{A}_3 (N)$ &
$40881.270\ldots$   &  $40881.008\ldots$  &  $40880.998\ldots$  \\
\hline 
\bottomrule
\end{tabular}
\caption{Numerical results for $\a_{3,1} (5) = 40881$.}
\label{tab:a31_5}
\end{table}

Our results in Section \ref{sec6} give upper bounds for the error in $\mathcal{A}_1 (N)+\mathcal{A}_2 (N)+\mathcal{A}_3 (N)$ by $O( N^{-\frac{3}{2}} \log (N)^2 )$ for fixed $n$. This should be compared with the $O( N^{-\frac{3}{2}} )$ error that one would have in the case of an ordinary modular form of the same weight. Despite that, our numerical results suggest that $\mathcal{A}_1 (N)+\mathcal{A}_2 (N)+\mathcal{A}_3 (N)$ converges faster than $\mathcal{A}_1 (N)$ due to cancellations between $\mathcal{A}_1 (N)$, $\mathcal{A}_2 (N)$, and $\mathcal{A}_3 (N)$. It would be interesting to go beyond numerical analysis, understand whether this is in fact the case and whether there is another representation of the Fourier coefficients that can make this behavior obvious.


\begin{thebibliography}{99}
		
	\bibitem{Alexandrov:2016enp}
S.~Alexandrov, S.~Banerjee, J.~Manschot, and B.~Pioline, {\it {Indefinite theta
  series and generalized error functions}},
  \href{http://xxx.lanl.gov/abs/1606.05495}{{\tt arXiv:1606.05495}}.

\bibitem{Alexandrov:2016tnf}
S.~Alexandrov, S.~Banerjee, J.~Manschot, and B.~Pioline, {\it {Multiple
  D3-instantons and mock modular forms I}},  {\em Commun. Math. Phys.} {\bf
  353} no.~1 (2017), 379--411, \href{http://xxx.lanl.gov/abs/1605.05945}{{\tt
  arXiv:1605.05945}}.

\bibitem{Alexandrov:2017qhn}
S.~Alexandrov, S.~Banerjee, J.~Manschot, and B.~Pioline, {\it {Multiple
  D3-instantons and mock modular forms II}},
  \href{http://xxx.lanl.gov/abs/1702.05497}{{\tt arXiv:1702.05497}}.

\bibitem{AP}
T.~Apostol, {\em Modular functions and {D}irichlet series in number theory},
  vol.~41 of {\em Graduate Texts in Mathematics}.
\newblock Springer-Verlag, New York, second~ed., 1990.

\bibitem{BKM}
K.~Bringmann, J.~Kaszian, and A.~Milas, {\em {Vector-valued higher depth
  quantum modular form and higher Mordell integrals}},
\href{http://xxx.lanl.gov/abs/1803.06261}{\tt arXiv:1803.06261}.

\bibitem{bringmann2016}
K.~Bringmann, J.~Kaszian, and L.~Rolen, {\it {Indefinite theta functions arising in Gromov-Witten Theory of elliptic orbifolds}}, Cambridge Journal of Mathematics, accepted for publication, \href{http://xxx.lanl.gov/abs/1608.08588}{\tt arXiv:1608.08588}.

\bibitem{Bringmann:2010sd}
K.~Bringmann and J.~Manschot, {\it {From sheaves on $\IP^2$ to a generalization
  of the Rademacher expansion}}, American Journal of Mathematics \textbf{135} (2013), 1039--1065, \href{http://xxx.lanl.gov/abs/1006.0915}{\tt arXiv:1006.0915}.

\bibitem{CohenStromberg}
H.~Cohen and F.~Str\"omberg, {\em Modular forms: a classical approach}, vol.~179 of {\em Graduate
  Studies in Mathematics}.
\newblock American Mathematical Society, Providence, RI, 2017.

\bibitem{funke2017theta}
J.~Funke and S.~Kudla, {\it {Theta integrals and generalized error functions,
  II}},  \href{http://xxx.lanl.gov/abs/1708.02969}{{\tt arXiv:1708.02969}}.

\bibitem{HardyRamanujan2}
G.~Hardy and S.~Ramanujan, {\it Asymptotic formul\ae \ in combinatory
  analysis [{P}roc. {L}ondon {M}ath. {S}oc. (2) {\bf 17} (1918), 75--115]},  in
  {\em Collected papers of {S}rinivasa {R}amanujan}, pp.~276--309.
\newblock AMS Chelsea Publ., Providence, RI, 2000.

\bibitem{HardyRamanujan1}
G.~Hardy and S.~Ramanujan, {\it Une formule asymptotique pour le nombre des
  partitions de {$n$} [{C}omptes {R}endus, 2 {J}an. 1917]},  in {\em Collected
  papers of {S}rinivasa {R}amanujan}, pp.~239--241.
\newblock AMS Chelsea Publ., Providence, RI, 2000.

\bibitem{klyachko1991}
A.~Klyachko, {\it Moduli of vector bundles and numbers of classes},  {\em
  Funktsional. Anal. i Prilozhen.} {\bf 25} no.~1 (1991), 81--83.

\bibitem{kool2009}
M.~Kool, {\it Euler characteristics of moduli spaces of torsion free sheaves on
  toric surfaces},  {\em Geom. Dedicata} {\bf 176} (2015), 241--269.

\bibitem{Kudla2018}
S.~Kudla, {\it Theta integrals and generalized error functions},  {\em
  manuscripta mathematica} {\bf 155} (Mar, 2018), 303--333.

\bibitem{Manschot:2010nc}
J.~Manschot, {\it {The Betti numbers of the moduli space of stable sheaves of
  rank 3 on $\mathbb{P}^2$}},  {\em Lett. Math. Phys.} {\bf 98} (2011), 65--78,
  [\href{http://xxx.lanl.gov/abs/1009.1775}{{\tt arXiv:1009.1775}}].

\bibitem{Manschot:2014cca}
J.~Manschot, {\it {Sheaves on $\mathbb{P}^2$ and generalized Appell
  functions}},  {\em Adv. Theor. Math. Phys.} {\bf 21} (2017), 655--681,
  [\href{http://xxx.lanl.gov/abs/1407.7785}{{\tt arXiv:1407.7785}}].

\bibitem{Manschot:2017xcr}
J.~Manschot, {\it {Vafa-Witten theory and iterated integrals of modular
  forms}},  \href{http://xxx.lanl.gov/abs/1709.10098}{{\tt arXiv:1709.10098}}.

\bibitem{Montonen:1977sn}
C.~Montonen and D.~Olive, {\it {Magnetic Monopoles as Gauge Particles?}},
  {\em Phys. Lett.} {\bf 72B} (1977), 117--120.

\bibitem{Nazaroglu:2016lmr}
C.~Nazaroglu, {\it {$r$-Tuple Error Functions and Indefinite Theta Series of
  Higher-Depth}},  \href{http://xxx.lanl.gov/abs/1609.01224}{{\tt
  arXiv:1609.01224}}.

\bibitem{Osborn:1979tq}
H.~Osborn, {\it {Topological Charges for N=4 Supersymmetric Gauge Theories and
  Monopoles of Spin 1}},  {\em Phys. Lett.} {\bf 83B} (1979), 321--326.

\bibitem{Rademacher37}
H.~Rademacher, {\it On the {P}artition {F}unction p(n)},  {\em Proc. London
  Math. Soc. (2)} {\bf 43} no.~4 (1937),  241--254.

\bibitem{Rademacher43}
H.~Rademacher, {\it On the expansion of the partition function in a series},
  {\em Ann. of Math. (2)} {\bf 44} (1943), 416--422.

\bibitem{Vafa:1994tf}
C.~Vafa and E.~Witten, {\it {A Strong coupling test of S duality}},  {\em Nucl.
  Phys.} {\bf B431} (1994), 3--77,
  [\href{http://xxx.lanl.gov/abs/hep-th/9408074}{{\tt hep-th/9408074}}].

\bibitem{weist2009}
T.~Weist, {\it Torus fixed points of moduli spaces of stable bundles of rank
  three},  {\em J. Pure Appl. Algebra} {\bf 215} no.~10 (2011), 2406--2422.

\bibitem{westerholt2016}
M.~Westerholt-Raum, {\it {Indefinite theta series on tetrahedral cones}},
  \href{http://xxx.lanl.gov/abs/1608.08874}{{\tt arXiv:1608.08874}}.

\bibitem{Witten:1978mh}
E.~Witten and D.~Olive, {\it {Supersymmetry Algebras That Include
  Topological Charges}},  {\em Phys. Lett.} {\bf 78B} (1978), 97--101.

\bibitem{yoshioka1994betti}
K.~Yoshioka, {\it The {B}etti numbers of the moduli space of stable sheaves of
  rank {$2$} on {$\mathbb P^2$}},  {\em J. Reine Angew. Math.} {\bf 453} (1994),
  193--220.

\bibitem{yoshioka1995betti}
K.~Yoshioka, {\it The {B}etti numbers of the moduli space of stable sheaves of
  rank {$2$} on a ruled surface},  {\em Math. Ann.} {\bf 302} no.~3 (1995), 
  519--540.

\bibitem{zagier75}
D.~Zagier, {\it Nombres de classes et formes modulaires de poids {$3/2$}},
  {\em C. R. Acad. Sci. Paris S\'er. A-B} {\bf 281} (1975), no.~21 Ai,
  A883--A886.

\bibitem{zwegers2008}
S.~Zwegers, {\it {Mock theta functions}}, Ph.D. thesis, Universiteit Utrecht, 2002.

\end{thebibliography}
\end{document}